\documentclass[final,1p,times]{elsarticle}

\usepackage{amssymb}
\usepackage{amsthm}

\usepackage{latexsym}
\usepackage{amsmath}
\usepackage{color}
\usepackage{graphicx}
\usepackage{indentfirst}
\usepackage{mathrsfs}
\usepackage{pdfsync}
\usepackage{hyperref}

\hypersetup{hypertex=true,
            colorlinks=true,
            linkcolor=blue,
            anchorcolor=blue,
            citecolor=blue}
\allowdisplaybreaks
\usepackage[utf8]{inputenc}

\newtheorem{theorem}{\color{black}\indent Theorem}
\newtheorem{lemma}{\color{black}\indent Lemma}[section]
\newtheorem{proposition}{\color{black}\indent Proposition}
\newtheorem{definition}{\color{black}\indent Definition}[section]
\newtheorem{remark}{\color{black}\indent Remark}[section]

\journal{a}

\begin{document}

\begin{frontmatter}



\title{Kolmogorov's Theorem for Degenerate Hamiltonian Systems with  Continuous Parameters}


\author{{Jiayin Du$^{a,*}$ \footnote{ E-mail address : dujy668@nenu.edu.cn} ,~Yong Li$^{a,b}$ \footnote{ E-mail address : liyong@jlu.edu.cn},~Hongkun Zhang$^{c}$ \footnote{E-mail address : hongkunz@umass.edu\\\indent ~*Corresponding author}
}\\
{$^{a}$College of Mathematics and Statistics, and Center for Mathematics and Interdisciplinary Sciences, Northeast Normal
University, Changchun, 130024, P. R. China.}\\
{$^{b}$College of Mathematics, Jilin University, Changchun, 130012, P. R. China.}\\
{$^{c}$Department of Mathematics and Statistics, University of Massachusetts, Amherst, 01003, USA.} }

\begin{abstract}
In this paper, we investigate Kolmogorov type theorems for small perturbations of degenerate Hamiltonian systems. These systems are index by a parameter $\xi$ as \( H(y,x,\xi) = \langle\omega(\xi),y\rangle + \varepsilon P(y,x,\xi,\varepsilon) \) where $\varepsilon>0$. We assume that the frequency map, $\omega$, is continuous with respect to $\xi$. Additionally, the perturbation function, $P(y,x,\cdot, \varepsilon)$, maintains H\"{o}lder continuity about $\xi$. We prove that persistent invariant tori retain the same frequency as those of the unperturbed tori, under certain topological degree conditions and a weak  convexity condition for the frequency mapping. Notably, this paper presents, to our understanding, pioneering results on the KAM theorem under such conditions-with only assumption of continuous  dependence of frequency mapping $\omega$ on the parameter.
\end{abstract}






\end{frontmatter}



\tableofcontents

\section{Introduction}\label{sec:1}
\setcounter{equation}{0}
\setcounter{definition}{0}
\setcounter{proposition}{0}
\setcounter{lemma}{0}
\setcounter{remark}{0}
As a conservation law of energy, Hamiltonian systems are frequently considered to describe models arose in celestial mechanics or the motion of charged particles in magnetic fields, see \cite{chierchia,meyer3,wayne}.

The classical KAM theory, as presented by Arnold, Kolmogorov, and Moser \cite{arnold,kolmogorov,moser}, posits that under the Kolmogorov non-degenerate condition, most invariant tori of an integrable Hamiltonian system can withstand small perturbations. While these tori might undergo minor deformations, they transform into other invariant tori that retain the original frequency





Numerous methods have been explored to study the persistence of invariant tori and the preservation of toral frequency within Hamiltonian systems under certain non-degenerate conditions. For instance, the KAM approach was used in \cite{benettin,bounemoura,de,li,poschel,salamon}. The direct method using Lindstedt series can be found in references \cite{chierchia1,eliasson,gallavotti}, while renormalization group techniques were discussed in \cite{bricmont,gallavotti1}. Notably, the study presented in \cite{chow} introduced the idea of partial preservation of unperturbed frequencies and delved into the persistence problem on a specified smooth sub-manifold for real analytic Hamiltonian systems, particularly under the R\"{u}ssmann-like non-degenerate condition. For insights under analogous conditions,  see also \cite{sevryuk}.






Yet, in the context of persistence, two fundamental questions emerge that warrant attention:

\noindent \textbf{Q1}: In the event of a failure in the Kolmogorov non-degenerate condition, can the invariant tori with the same frequency still be preserved under small perturbations?

\noindent \textbf{Q2}: If the regularity of the frequency mapping diminishes to mere continuity, can the aforementioned result withstand small perturbations?

To shed light on these questions, we review previous findings and offer a more comprehensive overview.

\subsection{Degeneracy}
Consider the real analytic nearly integrable Hamiltonian system\begin{align}\label{eq1}
H(y,x,\varepsilon )=h(y)+\varepsilon P(y,x,\varepsilon),
\end{align}
where $x$ is the angle variable in  the standard torus $\mathbb{T}^{n}$,  $n$ refers to the dimension; $y$ is the action variable in a bounded closed region $G\subset \mathbb{R}^n$, and $\varepsilon>0$ is a small parameter.

A fundamental assumption in historical research is the Kolmogorov non-degenerate condition. However, if we assume that there exists a  $y_0\in G$ such that,
\begin{align*}
\det\frac{\partial^2h(y_0)}{\partial y^2}=0,
\end{align*}
then the Kolmogorov condition is not satisfied. The spatial solar system serves as a prominent example of this situation, as detailed in \cite{fejoz}. Naturally, a question arises: does the persistence result still stand under these conditions? This question has been a primary motivation for this research.


In fact, even under weaker non-degenerate conditions, KAM tori might not preserve their frequencies. As demonstrated in \cite{bruno,russmann2,sevryuk1}, under the Brjuno non-degenerate condition and R\"{u}ssmann non-degenerate condition, the presumption of an unchanged frequency may not necessarily hold true. This is because the frequency of perturbed tori can undergo slight variations. Similar observations are noted in \cite{biasco,cheng,chow,cong,heinz,qian1,xu4}. Consequently, deriving conditions that assure the persistence of frequencies for KAM tori in the context of a degenerate Hamiltonian becomes rather challenging. Furthermore, the issue of the perturbed invariant tori maintaining a consistent frequency has seldom been tackled for degenerate systems.

\subsection{Regularity}

On the matter of regularity, it's worth noting the distinctions in the studies of various researchers. Kolmogorov \cite{kolmogorov} and Arnold \cite{arnold} focused on real analytic Hamiltonian systems. In contrast, Moser \cite{moser} illustrated that Hamiltonian systems don't necessarily need to be analytic; a high, albeit finite, level of regularity for the Hamiltonian suffices. This regularity requirement was later reduced to \(C^5\) in work by \cite{russmann1}. Further important contributions on this topic can be found in \cite{bounemoura,koud,herman,salamon}.
Moreover, the scenario where the frequency mapping has Lipschitz continuous parameters has been explored in \cite{poschel1}. A subsequent question of interest is: what are the implications when the regularity of the frequency mapping is merely continuous with respect to its parameters?

More precisely, we consider a family of Hamiltonian systems under small perturbations:
\begin{equation}\label{eqb1}
H(y,x,\xi,\varepsilon)=\langle\omega(\xi),y\rangle+\varepsilon P(y,x,\xi,\varepsilon),
\end{equation}
where $(y,x)\in G\times\mathbb{T}^n$ and $\xi$ is a parameter in a bounded closed region $O\subset \mathbb{R}^n$. The function $\omega(\cdot)$ is continuous with respect to $\xi$ on $O$. The function $P(\cdot,\cdot,\xi,\varepsilon)$ is real analytic with respect to $y$ and $x$, and $P(y,x,\cdot,\varepsilon)$ is H\"{o}lder continuous with respect to the parameter $\xi$ with H\"{o}lder index $\beta$, for some $0<\beta<1$. Additionally, $\varepsilon>0$ is a small parameter.

It's important to note that in the conventional KAM iteration process, the regularity of the frequency mapping concerning the parameters must be at least Lipschitz continuous. This ensures that the parameter domain remains intact. However, when the regularity of the frequency mapping is less stringent than Lipschitz continuous, the traditional method of parameter excavation becomes infeasible. This necessitates the exploration of novel approaches to address the issue.


\subsection{Our work}
\indent\indent

Regarding regularity, when the frequency mapping is continuous with respect to parameters, we prove  that the perturbed invariant tori retain the same Diophantine frequency as their unperturbed counterparts for Hamiltonian systems as described in (\ref{eqb1}), see Theorem \ref{th1}. For the degeneracy problem, persistence results under the highly degenerate Hamiltonian system (\ref{eq1}) are proved in Theorem \ref{th2}.

We establish sufficient conditions based on the topological degree condition $\rm(A0)$ and the weak convexity condition $\rm(A1)$ for frequency mapping. Detailed descriptions of these conditions are provided in Section 2. In deriving our primary results, we employ the quasi-linear KAM iteration procedure as in  \cite{chow,han,li1,qian2}. Notably, we introduce a  parameterized family Hamiltonian systems to counteract frequency drift. Specifically, we adjust the action variable to maintain constant frequency for the highly degenerate Hamiltonian system (\ref{eq1}).   It's also noteworthy that the weak convexity condition proposed in this paper is necessary  regardless of the smoothness level of the frequency mapping, as evidenced by Proposition \ref{pro1}.

{It should be pointed out that the KAM-type theorems associated with parameter family are due to Moser \cite{moser1967}, P\"{o}schel\cite{poschel1989}. However, our results are different from theirs: a Diophantine frequency can be given in advance, but Moser's systems need to be modified in KAM iteration and hence cannot be given beforehand; in P\"{o}schel's approach, the frequency set need to be dug out in KAM process. Our method is to find a parameter in the family of systems by translating parameter. Of course, it does not work generally. As pointed out in our paper, the weak convexity condition (A1) is indispensable. To our knowledge, this setting seems to be first.}

The rest of this paper is organized as follows. In Section \ref{sec:b2}, we state our main results (Theorems \ref{th1}, \ref{th2}, \ref{pro2} and \ref{th3}). We will describe the quasi-linear iterative scheme,  show the detailed construction and estimates for one cycle of KAM steps in Section \ref{sec:b3}. In Section \ref{sec:b4}, we complete the proof of Theorem \ref{th1} by deriving an iteration lemma and showing the convergence of KAM iterations. In Section \ref{sec:b5}, we prove Theorem \ref{th2} , which covers the analytic situation, and is also a special case  of Theorem \ref{th1}. We also  prove Theorem \ref{th3} by directly computing. Finally, the proof of Theorem \ref{pro2} can be found in Appendix B.

\section{Main results}\label{sec:b2}

To state our main results we need first to introduce a few definitions and notations.
\begin{description}
\item[\rm{(1)}] Given a domain $D\subset G\times \mathbb{T}^n$, we let  $\bar{D}$, $\partial D$ denote the closure of $D$ and the boundary of $D$,  respectively. $D^o:=\bar{D}\setminus\partial D$ refers to the interior.
\item[\rm(2)] We shall use the same symbol $\vert\cdot\vert$ to denote an equivalent vector norm and its induced matrix norm, absolute value of functions, etc, and use $\vert\cdot\vert_D$ to denote the supremum norm of functions on a domain $D$.
\item[\rm(3)] For the perturbation function $P(y,x,\xi)$, which is analytic about $y$ and $x$ and H\"{o}lder continuous about $\xi$ with H\"{o}lder index $\beta$, $0<\beta<1$, we define its norm as follows
  $$\vert\|P\|\vert_{D}=\|P\|_{D}+\|P\|_{C^\beta}$$ where
  \begin{equation}\label{beta}\|P\|_{C^\beta}=\sup_{\xi\neq\zeta,~\xi,\zeta\in {O}}\frac{\vert P(y,x,\xi)-P(y,x,\zeta)\vert}{\vert\xi-\zeta\vert^\beta},~~~\forall (y,x)\in {D}.\end{equation}
\item[\rm(4)] For any two complex column vectors $\xi$, $\eta$ in the same space, $\langle\xi,\eta\rangle$ always stands for $\xi^\top\eta$.
\item[\rm(5)] $i_d$ is the unit map, and $I_d$ is the unit matrix.
\item[\rm(6)] For a vector value function $f$, $Df$ denotes the Jacobian matrix of $f$, and $J_f=det Df$ its Jacobian determinant.
\item[\rm(7)] All Hamiltonian in the sequel are endowed with the standard symplectic structure.
\item[\rm(8)] As pointed out in \cite{poschel}, the real  analyticity of the Hamiltonian $H(y,x)$ about $y$ and $x$ on $G\times{\mathbb{T}^n}$  implies that the analyticity extends to a complex neighbourhood $D(s,r)$ of $G\times{\mathbb{T}^n}$, where $D(s,r)$ is defined for some $0<s,r<1$, with
$$D(s,r):=\{(y,x):dist(y,G)<s,\vert\textrm{Im}x\vert<r\}.$$ 
\item[\rm(9)] For $\forall \delta>0$, $y_0\in G$, let
\begin{align*}
B_\delta(y_0)&=:\{y\in G:\vert y-y_0\vert<\delta\},\\
\bar B_\delta(y_0)&=:\{y\in G:\vert y-y_0\vert\leq\delta\}.
\end{align*}

\end{description}

Let $\Omega\subset \mathbb{R}^n$ be a bounded and open domain. We first give the definition of the degree for $f\in C^2(\bar\Omega,\mathbb{R}^n)$, see \cite{guo}.
\begin{definition}
If $f\in C^2(\bar\Omega,\mathbb{R}^n)$, and $p\in \mathbb{R}^n\setminus f(\partial\Omega)$, let $\varsigma=\inf_{x\in\partial\Omega}\left\|f(x)-p\right\|$.

\begin{description}
\item[\rm(1)] Denote $N_f=\left\{x\vert x\in\Omega,J_f(x)=0\right\}$. If $p\notin f(N_f)$, then
\begin{align*}
\deg(f,\Omega,p):=\sum_{x\in { f^{-1}(p)}} sign\left(J_{f}(x)\right),
\end{align*}
setting $\deg\left(f,\Omega,p\right)=0$ if $f^{-1}\left(p\right)={\O}$.
\item[\rm(2)] If $p\in f(N_f)$, by Sard theorem (see \cite{guo}),  $\deg\left(f,\Omega,p\right)$ is  defined as $\deg(f,\Omega,p):=\deg(f,\Omega,p_1)$, for some (and all) $p_1\notin f(N_f)$ such that $\left\|p_1-p\right\|<\frac{\varsigma}{7}$.
\end{description}
\end{definition}
We next extend the definition of degree to the general continuous mapping, see \cite{guo}.
\begin{definition}(Brouwer's Degree)
Let $g\in C(\bar\Omega,\mathbb{R}^n)$, and $p\in \mathbb{R}^n\setminus g(\partial\Omega)$, i.e., $\varsigma_*=:\inf_{x\in\partial\Omega}\left\|g(x)-p\right\|>0$.  Let
$$S=:\left\{f\Big\vert f\in C^2\left(\bar\Omega,\mathbb{R}^n\right),  \max_{x\in\bar\Omega}\left\{\left\|g(x)-f(x)\right\|<\varsigma_*\right\}\right\}.$$
Define the degree of $g$ as :
$$\deg{(g,\Omega,p)}:=\deg{(f,\Omega,p)},~~~\forall f\in S.$$
\end{definition}
We are now ready to state our assumptions. Mainly we
consider (\ref{eqb1}), i.e., for any $\varepsilon>0$ small enough,  we consider the parameterized family of perturbed Hamiltonian equations
\begin{equation*}
\left\{
\begin{array}{ll}
H:G\times \mathbb{T}^n\times O\rightarrow \mathbb{R}^1,\\
H(y,x,\xi)=\langle\omega(\xi),y\rangle+\varepsilon P(y,x,\xi,\varepsilon).
\end{array}
\right.
\end{equation*}

First, we make the following assumptions:
\begin{description}
\item[\rm(A0)]
Fix $\xi_0\in O^o$ such that
\begin{align}\label{A0}
&\deg\left(\omega(\cdot), O^o, \omega(\xi_0)\right)\neq0.
\end{align}
\item[\rm(A1)] There are $\sigma>0$, $0<L\leq\beta$, ($\beta$ was defined in (\ref{beta})), such that
\begin{align}\label{A1}
 \left\vert\omega(\xi)-\omega(\xi_{*})\right\vert\geq \sigma\left\vert\xi-\xi_{*}\right\vert^L,~~~~\forall \xi,\xi_{*}\in O. \end{align}
\item[\rm(A2)] For the given $\xi_0\in O^o$, $\omega(\xi_0)$ satisfies the Diophantine condition
\begin{align}\label{eqa2}
\left\vert\langle k,\omega(\xi_0)\rangle\right\vert>\frac{\gamma}{\vert k\vert^\tau},~~~k\in{\mathbb{Z}^n\setminus{\{0\}}},
\end{align}
where $k=(k_1,\cdots,k_n)$, $\vert k\vert=\left\vert k_1\right\vert+\cdots+\left\vert k_n\right\vert$, $\gamma>0$ and $\tau> n-1$.
\end{description}

Then, we have the following main results:
\begin{theorem}\label{th1}
Consider Hamiltonian system {\rm(\ref{eqb1})}.
Assume that ${\rm(A0)}$ , ${\rm(A1)}$ and ${\rm(A2)}$ hold.
{{Then there exists a sufficiently small  $\varepsilon_0>0$,  for any $0<\varepsilon<\varepsilon_0$, there exist $\xi_\varepsilon\in O$ and {a symplectic transformation $\Psi_*$ such that
$$H(\Psi_*(y,x,\xi_\varepsilon),\varepsilon)=e_*+\langle\omega(\xi_0),y\rangle+\bar h_*(y,\xi_{\varepsilon})+P_*(y,x,\xi_\varepsilon,\varepsilon)$$
where $e_*$ is a constant, $\bar h_*(y,\xi_{\varepsilon})=O(\vert y\vert^2)$, $P_*=O(\vert y\vert^2)$.}
Thus the perturbed Hamiltonian system $H(y,x,\xi_\varepsilon,\varepsilon)$ admits an invariant torus with frequency $\omega(\xi_0)$.}}
\end{theorem}
\begin{remark}
It should be emphasized that we deal with the degenerate Hamiltonian system in which the frequency mapping is continuous about parameters and the perturbation is H\"{o}lder continuous about parameters in this theorem. It seems to be the first version in KAM theory.
\end{remark}

In the following, we will give some examples to state that conditions ${\rm(A0)}$ and ${\rm(A1)}$ are indispensable, especially for condition ${\rm(A1)}$. See below for a counter example:
\begin{proposition}\label{pro1}
Consider the Hamiltonian system {\rm(\ref{eqb1})}, for $n=2$,  with $$\omega(\xi)=(\omega_1(\xi_1),\omega_2(\xi_2))^\top,~~~\varepsilon P=P_0(\varepsilon)y_2,$$
where
\begin{align*}
&\omega_1(\xi_1)=\bar{\omega}_1+\xi_1,~~\xi_1\in(-1,1),\\
&\omega_2{(\xi_2)}=\left\{\begin{array}{lll}
\bar{\omega}_2+\exp\{-\frac{1}{(\xi_2+{\frac{1}{2}})^2}\},~~~&\xi_2\in(-1,-\frac{1}{2}),\\
\bar{\omega}_2,~~~&\xi_2\in[-\frac{1}{2},\frac{1}{2}],\\
\bar{\omega}_2-\exp\{-\frac{1}{(\xi_2-{\frac{1}{2}})^2}\},~~~&\xi_2\in(\frac{1}{2},1),
                           \end{array}\right.
\end{align*}
$\bar{\omega}=(\bar{\omega}_1,\bar{\omega}_2)^\top$ satisfies Diophantine condition {\rm(\ref{eqa2})}, and
$$P_0(\varepsilon)=\left\{\begin{array}{lll}
                              0,&&\varepsilon=0, \\
                              \varepsilon^\ell\sin\frac{1}{\varepsilon},&&\varepsilon\neq0,\,\ell\in \mathbb{Z}^+\setminus\{0\}.
                            \end{array}\right.$$
Then condition $\rm(A1)$ fails for any parameter $\xi\in(-1,1)$. Moreover, Theorem \ref{th1} fails.
\end{proposition}
See Appendix A for the complete proof.

\begin{remark}
This counter example implies that $\rm(A1)$ is necessary no matter how smooth the frequency mapping $\omega(\xi)$ is.
\end{remark}

Nevertheless, one asks what happens to the frequency mapping in the analytic situation.  As a special case of our Theorem \ref{th1}, we also obtain the Kolmogorov's theorem for analytic Hamiltonian systems under degenerate conditions. This is stated in the following theorem.

\begin{theorem}\label{th2}
Consider the real analytic Hamiltonian system  {\rm(\ref{eq1})}. Fix $\xi_0\in G$ such that $\rm(A0)$, $\rm(A1)$ and $\rm(A2)$ hold for $\omega(\xi)=\nabla h(\xi)$, $O=G$, and $L>0$.
Then there exist a sufficiently small positive constant $\varepsilon'>0$ such that if  $0<\varepsilon<\varepsilon'$, there exists $y_\varepsilon\in G$ such that Hamiltonian system {\rm(\ref{eq1})} at $y=y_\varepsilon$ admits an invariant torus with frequency $\nabla h(\xi_0)$.
\end{theorem}
This theorem is proved in Section 5.1.

Next, we will give an example that satisfies  conditions $\rm(A0)$-$\rm(A1)$.  For simplicity we use the action variable $y$ as the parameter $\xi$.
\begin{theorem}\label{pro2}
Consider the Hamiltonian system {\rm(\ref{eq1})} with
\begin{equation*}
h(y)=\langle \omega,y\rangle+\frac{1}{2l+2}\vert y\vert^{2l+2},
\end{equation*}
where $y\in G\subset \mathbb{R}^n$, $l$ is a positive integer, $\omega\in \mathbb{R}^n\setminus{\{0\}}$ satisfies the Diophantine condition {\rm(\ref{eqa2})}.
Then the Hamiltonian system {\rm(\ref{eq1})} admits an invariant torus with frequency $\omega$ for any small enough perturbation.
\end{theorem}

The proof can be found in Appendix B.
\begin{proposition}\label{cor1}
If $h(y)=\langle \omega,y\rangle+\frac{1}{2l+1}\left\vert y\right\vert^{2l+1}$ in Hamiltonian system {\rm(\ref{eq1})}, $\omega\in \mathbb{R}^n\setminus{\{0\}}$ satisfies the Diophantine condition {\rm(\ref{eqa2})}, then the system may not admit torus with frequency $\omega$.
\end{proposition}
The proof can be found in Appendix C.

Above results  imply that condition ${\rm(A0)}$ is indispensable for $n>1$ case.
Furthermore, we also prove that for $n=1$, the persistence results in Theorem \ref{pro2} hold under some weaker conditions, provided that the frequency satisfies Diophantine condition ${\rm(A2)}$.
\begin{theorem} \label{th3}
Consider Hamiltonian {\rm(\ref{eq1})} with
\begin{align*}
h(y)=\omega y+g(y),~~~~\varepsilon P(y,x,\varepsilon)=\varepsilon P(y),
\end{align*}
where $y\in G=[-1,1]\subset \mathbb{R}^1$, $\omega$ satisfies Diophantine condition {\rm(\ref{eqa2})}.
\begin{description}
\item[\rm(1)] If $g(y)\in C^{2\ell+1}$, $g'(0)=\cdots=g^{2\ell}(0)=0$, $g^{2\ell+1}(0)\neq 0$, $\ell$ is a positive integer, then the perturbed system admits at least two invariant tori with frequency $\omega$ for the small enough perturbation satisfying $\varepsilon P'(y)\, sign (g^{2\ell+1}(0))<0$; conversely, if $\varepsilon P'(y)\, sign (g^{2\ell+1}(0))>0$, the unperturbed invariant torus with frequency $\omega$ will be destroyed.
\item[\rm(2)] If $g(y)\in C^{2\ell+2}$, $g'(0)=\cdots=g^{2\ell+1}(0)=0$, $g^{2\ell+2}(0)\neq 0$, $\ell$ is a positive integer, then the perturbed system admits an invariant tori with frequency $\omega$ for any small enough perturbation.
\end{description}
\end{theorem}
\begin{remark}
We don't know whether the results in Theorem \ref{th3} can be extended to higher dimensions or not.
\end{remark}
\section{KAM step}\label{sec:b3}
\indent In this section, we will describe the quasi-linear iterative scheme, show the detailed construction and estimates for one cycle of KAM steps, which is essential to study the KAM theory, see \cite{chow,han,li,li1,poschel}. It should be pointed out that in our KAM iteration, we present a new way to move parameters; while in the usual KAM iteration, one has to dig out a decreasing series of parameter domains, see \cite{chow,han,li1,poschel1,poschel,qian2,qian1}.
\subsection{Description of the 0-th KAM step.}
\indent Given an integer $m>L+1$, where $L$ was defined as in $\rm(A1)$. Denote $\rho=\frac{1}{2(m+1)}$, and let $\eta>0$ be an integer such that $(1+\rho)^\eta>2$. We define
\begin{equation}\label{gamma}\gamma=\varepsilon^{\frac{1}{4(n+m+2)}}.\end{equation}
Consider the perturbed Hamiltonian (\ref{eqb1}). We first define the following $0$-th KAM step parameters:
\begin{align}\label{eqb2}
&r_0=r,~~~~~\gamma_0=\gamma,~~~~~e_0=0,
~~~~~\bar{h}_0=0,~~~~~\mu_0=\varepsilon^{\frac{1}{8\eta(\tau+1)(m+1)}},\\
&s_0=\frac{s\gamma_0}{16(M^*+2)K_1^{\tau+1}},~O_0=\{\xi\in O\vert~\vert\xi-\xi_0\vert<dist(\xi_0,\partial O)\},\notag\\
&D(s_0,r_0):=\{(y,x):dist(y,G)<s_0,\vert\textrm{Im}x\vert<r_0\},\notag
\end{align}
where $0<s_0,\gamma_0,\mu_0\leq 1$, $\tau>n-1$, $M^*>0$ is a constant defined as in Lemma \ref{le2}, and
\begin{align*}
K_1=([\log\frac{1}{\mu_0}]+1)^{3\eta}.
\end{align*}

Therefore, we can write
\begin{align*}
H_0&=: H(y,x,\xi_0)=N_0+P_0,\\
N_0&=: N_0(y,\xi_0,\varepsilon)=e_0+\langle\omega(\xi_0),y\rangle+\bar{h}_0,\\
P_0&=:\varepsilon P(y,x,\xi_0,\varepsilon).
\end{align*}

We first prove an important estimate.
\begin{lemma}
\begin{equation}\label{P0}
\vert\|P_0\|\vert_{D(s_0,r_0)}\leq\gamma_0^{n+m+2}s_0^m\mu_0.
\end{equation}
\end{lemma}
\begin{proof}
Using the fact $\gamma_0^{n+m+2}=\varepsilon^{\frac{1}{4}}$ and $[\log\frac{1}{\mu_0}]+1<\frac{1}{\mu_0}$, we have
\begin{align*}
s_0^m=\frac{s^m\varepsilon^{\frac{m}{4(n+m+2)}}}{16^m(M^*+2)^mK_1^{m(\tau+1)}}
>\frac{s^m\varepsilon^{\frac{m}{4(n+m+2)}}\mu_0^{3\eta m(\tau+1)}}{16^m(M^*+2)^m}\geq\frac{s^m\varepsilon^{\frac{m}{4(n+m+2)}+\frac{3}{8}}}{16^m(M^*+2)^m}.
\end{align*}
Moreover, let $\varepsilon_0>0$ be small enough so that
\begin{equation}\label{vare0}
\varepsilon_0^{\frac{1}{8}-\frac{1}{8\eta(\tau+1)(m+1)}}\vert\|P\|\vert_{D(s_0,r_0)}\frac{16^m(M^*+2)^m}{s^m}\leq 1,
\end{equation}
using the fact that $\mu_0=\varepsilon^{\frac{1}{8\eta(\tau+1)(m+1)}}$,  we get
\begin{align}\label{u0}
\gamma_0^{n+m+2}s_0^m\mu_0&\geq\frac{s^m\varepsilon^{\frac{m}{4(n+m+2)}+\frac{3}{8}+\frac{1}{4}+\frac{1}{8\eta(\tau+1)(m+1)}}}{16^m(M^*+2)^m}
\geq\frac{s^m\varepsilon^{\frac{1}{4}+\frac{3}{8}+\frac{1}{4}+\frac{1}{8\eta(\tau+1)(m+1)}}}{16^m(M^*+2)^m}
\notag\\&=\varepsilon^{\frac{7}{8}}\frac{s^m\varepsilon^{\frac{1}{8\eta(\tau+1)(m+1)}}}{16^m(M^*+2)^m},
\end{align}
and by (\ref{vare0}) and $0<\varepsilon<\varepsilon_0$,
\begin{align*}
\varepsilon^{\frac{1}{8}-\frac{1}{8\eta(\tau+1)(m+1)}}\vert\|P\|\vert_{D(s_0,r_0)}\frac{16^m(M^*+2)^m}{s^m}\leq 1,
\end{align*}
i.e.,
\begin{align}\label{P}
\varepsilon^{\frac{1}{8}}\vert\|P\|\vert_{D(s_0,r_0)}\leq \frac{s^m\varepsilon^{\frac{1}{8\eta(\tau+1)(m+1)}}}{16^m(M^*+2)^m}.
\end{align}
Then by (\ref{u0}) and (\ref{P}),
\begin{align*}
\vert\|P_0\|\vert_{D(s_0,r_0)}=\varepsilon^{\frac{7}{8}} \varepsilon^{\frac{1}{8}}\vert\|P\|\vert_{D(s_0,r_0)}\leq \varepsilon^{\frac{7}{8}}\frac{s^m\varepsilon^{\frac{1}{8\eta(\tau+1)(m+1)}}}{16^m(M^*+2)^m}\leq\gamma_0^{n+m+2}s_0^m\mu_0,
\end{align*}
which implies (\ref{P0}).

The proof is complete.
\end{proof}

\subsection{Induction from  $\nu$-th KAM step}

\subsubsection{Description of the $\nu$-th KAM step}
We now define the $\nu$-th KAM step parameters:
$$r_\nu=\frac{r_{\nu-1}}{2}+\frac{r_0}{4},~~~s_\nu=\frac{1}{8}\mu_{\nu-1}^{2\rho}s_{\nu-1},~~~\mu_\nu=8^m\mu_{\nu-1}^{1+\rho},$$
where $\rho=\frac{1}{2(m+1)}$.

Now, suppose that at $\nu$-th step, we have arrived at the following real analytic Hamiltonian:
\begin{equation}\label{eqb3}
\begin{aligned}
H_\nu&=N_\nu+P_\nu,~~~~~~~~~~~~~~~~~~~~~~~~~~~~~~~~~~~~\\
N_\nu&=e_\nu+\langle\omega(\xi_0),y\rangle+\bar{h}_\nu(y,\xi),
\end{aligned}
\end{equation}
defined on $D(s_\nu,r_\nu)$
and
\begin{equation}\label{equ11}
\left\vert\left\|P_\nu\right\|\right\vert_{D(s_\nu,r_\nu)}\leq\gamma_0^{n+m+2}s_\nu^m\mu_\nu.
\end{equation}
The equation of motion associated to $H_{\nu}$ is
\begin{equation}\label{eqb4}
\left\{
\begin{array}{ll}
\dot{y}_\nu=-\partial_{x_\nu} H_\nu,\\
\dot{x}_\nu=~~\partial_{y_\nu} H_\nu.
\end{array}
\right.
\end{equation}

Except for additional instructions,{ we will omit the index for all quantities of the present KAM step (at $\nu$-th step) and use $+$ to index all quantities (Hamiltonian, domains, normal form, perturbation, transformation, etc.) in the next KAM step (at $(\nu+1)$-th step)}.
To simplify the notations, we will not specify the dependence of $P$, $P_+$ etc. All the constants $c_1$-$c_6$ below are positive and independent of the iteration process, and we will also use $c$ to denote any intermediate positive constant which is independent of the iteration process.

Define
\begin{align*}
r_+&=\frac{r}{2}+\frac{r_0}{4},\\
s_+&=\frac{1}{8}\alpha s,~~~~~\alpha=\mu^{2\rho}=\mu^{\frac{1}{m+1}},\\
\mu_+&=8^mc_0\mu^{1+\rho},~~~~~c_0=\max\{1,c_1,c_2,\cdots,c_6\},\\
K_+&=([\log\frac{1}{\mu}]+1)^{3\eta},\\
\hat{D}&=D(s,r_++\frac{7}{8}(r-r_+)),\\
\tilde{D}&=D(\frac{1}{2} s,r_++\frac{6}{8}(r-r_+)),\\
D(s)&=\{y\in C^n:\vert y\vert<s\},\\
D_{\frac{i}{8}\alpha}&=D(\frac{i}{8}\alpha s,r_++\frac{i-1}{8}(r-r_+)), ~~i=1,2,\cdots,8,\\
D_+&=D_{\frac{1}{8}\alpha}=D(s_+,r_+),\\
O_+&=\{\xi:dist(\xi,O)\leq \mu^\frac{1}{L}\},\\
\Gamma(r-r_+)&=\sum_{0<\vert k\vert\leq K_+}\left\vert k\right\vert^{3\tau+5}e^{-\vert k\vert\frac{r-r_+}{8}}.
\end{align*}

\subsubsection{Construct a symplectic transformation}
We will construct a symplectic coordinate transformation $\Phi_{+}$:
\begin{align}\label{Phi+}
\Phi_{+}:(y_{+},x_{+})\in D(s_{+},r_{+})\rightarrow \Phi_{+}(y_{+},x_{+})=(y,x)\in D(s,r)
\end{align}
 such that it transforms the Hamiltonian ($\ref{eqb3}$) into the Hamiltonian of the next KAM cycle (at $(\nu+1)$-th step), i.e.,
\begin{equation}\label{H+}
H_{+}=H\circ\Phi_{+}=N_{+}+ P_{+},
\end{equation}
where $N_{+}$ and $P_{+}$ have similar properties as $N$ and $P$ respectively on $D(s_{+},r_{+})$, and the equation of motion (\ref{eqb4}) is changed into
\begin{equation}\label{yx+}
\left\{
\begin{array}{ll}
\dot{y}_{+}=-\partial_{x_{+}} H_{+},\\
\dot{x}_{+}=~~\partial_{y_{+}} H_{+}.
\end{array}
\right.
\end{equation}
In the following, we prove (\ref{yx+}). Let $\Phi_+(y_+,x_+):=\left(\Phi_+^1(y_+,x_+),\Phi_+^2(y_+,x_+)\right)$, by (\ref{Phi+}), we have
\begin{align*}
\left(\begin{array}{c}
\dot{y}\\
\dot{x}
\end{array}\right)=\left(\begin{array}{cc}
(\partial_{y_+} \Phi_{+}^1)\dot{y}_{+}&(\partial_{x_+} \Phi_{+}^1)\dot{x}_{+}\\
(\partial_{y_+} \Phi_{+}^2)\dot{y}_{+}&(\partial_{x_+} \Phi_{+}^2)\dot{x}_{+}
\end{array}\right)=D\Phi_{+}\left(\begin{array}{c}
\dot{y}_{+}\\
\dot{x}_{+}
\end{array}\right),
\end{align*}
 by (\ref{Phi+}) and (\ref{H+}), we get
\begin{align*}
\left(\begin{array}{c}
\partial_{y_{+}} H\\
\partial_{x_{+}} H
\end{array}\right)&=\left(\begin{array}{cc}
\partial_{y} H\partial_{y_+} y&\partial_{x} H\partial_{y_+} x\\
\partial_{y} H\partial_{x_+} y&\partial_{x} H\partial_{x_+} x
\end{array}\right)
=\left(\begin{array}{cc}
\partial_{y_+} \Phi_+^1&\partial_{y_+} \Phi_+^2\\
\partial_{x_+} \Phi_+^1&\partial_{x_+} \Phi_+^2
\end{array}\right)
\left(\begin{array}{cc}
\partial_y H\\
\partial_x H
\end{array}\right)\\
&=D\Phi_{+}^\top\left(\begin{array}{cc}
\partial_y H\\
\partial_x H
\end{array}\right).
\end{align*}
Then this together with (\ref{eqb4}) yields
\begin{align*}
\left(\begin{array}{c}
\dot{y}_{+}\\
\dot{x}_{+}
\end{array}\right)&=D\Phi_{+}^{-1}\left(\begin{array}{c}
\dot{y}\\
\dot{x}
\end{array}\right)
=D\Phi_{+}^{-1}J\left(\begin{array}{cc}
\partial_y H\\
\partial_x H
\end{array}\right)
=D\Phi_{+}^{-1}J(D\Phi_{+}^{-1})^\top\left(\begin{array}{c}
\partial_{y_{+}} H\\
\partial_{x_{+}} H
\end{array}\right)\\
&
=J\left(\begin{array}{c}
\partial_{y_{+}} H\\
\partial_{x_{+}} H
\end{array}\right),\\
\end{align*}
where $J$ is the standard symplectic matrix, i.e.,
\begin{align*}
J=\left(\begin{array}{cc}
0&-I_d\\
I_d&0
\end{array}\right).
\end{align*}
This finishes the proof of   (\ref{yx+}).

Next, we show the detailed construction of $\Phi_+$ and the estimates of $P_+$.
\subsubsection{Truncation}\label{sub1}
Consider the Taylor-Fourier series of $P$:
\begin{equation*}
P=\sum_{k\in Z^n,~\imath\in Z_+^n}p_{k\imath}y^{\imath}e^{\sqrt{-1}\langle k,x\rangle},
\end{equation*}
and let $R$ be the truncation of $P$ of the form
\begin{equation*}
R=\sum_{\vert k\vert\leq K_+,~\vert\imath\vert\leq m}p_{k\imath}y^{\imath}e^{\sqrt{-1}\langle k,x\rangle}.
\end{equation*}

Next, we will prove that the norm of $P-R$ is much smaller than the norm of $P$ by selecting truncation appropriately, see the below lemma.
\begin{lemma}\label{leb1}
Assume that
\begin{align*}
\textbf{\textsc{(H1)}}: \int_{K_+}^{\infty}t^{n}e^{-t\frac{r-r_+}{16}}dt\leq\mu.
\end{align*}
Then there is a constant $c_1$ such that 
\begin{align}\label{equ9}
\left\vert\left\|P-R\right\|\right\vert_{D_\alpha}&\leq c_1\gamma_0^{n+m+2}s^m\mu^2,\\\label{equ10}
\left\vert\left\|R\right\|\right\vert_{D_\alpha}&\leq c_1\gamma_0^{n+m+2}s^m\mu.
\end{align}
\end{lemma}
\begin{proof}
Denote
\begin{align*}
I&=\sum_{\vert k\vert>K_+,~\imath\in Z_+^n}p_{k\imath}y^{\imath} e^{\sqrt{-1}\langle k,x\rangle},\\
II&=\sum_{\vert k\vert\leq K_+,~\vert\imath\vert> m}p_{k\imath}y^{\imath} e^{\sqrt{-1}\langle k,x\rangle}.
\end{align*}
Then
\begin{align*}
P-R=I+II.
\end{align*}
To estimate $I$, we note by (\ref{equ11}) that
\begin{align}\label{pki}
\left\vert\sum_{\imath\in Z_+^n}p_{k\imath}y^{\imath}\right\vert\leq \left\vert P\right\vert_{D(s,r)}e^{-\vert k\vert r}\leq\gamma_0^{n+m+2}s^{m}\mu e^{-\vert k\vert r},
\end{align}
where the first inequality has been frequently used in \cite{chow,cong,han,li,poschel,qian2,qian1,salamon} and the detailed proof see \cite{salamon}.
This together with $\textbf{\textsc{(H1)}}$ yields
\begin{align}
\left\vert I\right\vert_{\hat D}&\leq\sum_{\vert k\vert>K_+}\left\vert\sum_{\imath\in Z_+^n}p_{k\imath}y^{\imath}\right\vert e^{\vert k\vert(\frac{r_+}{8}+\frac{7r}{8})}
\leq\sum_{\vert k\vert>K_+}\left\vert P\right\vert _{D(s,r)}e^{-\vert k\vert\frac{r-r_+}{8}}\notag\\\label{equ7}
&\leq\gamma_0^{n+m+2}s^{m}\mu\sum_{\kappa=K_+}^{\infty}\kappa^{n}e^{-\kappa\frac{r-r_+}{8}}
\leq\gamma_0^{n+m+2}s^{m}\mu\int_{K_+}^{\infty}t^{n}e^{-t\frac{r-r_+}{16}}dt\\
&\leq\gamma_0^{n+m+2}s^{m}\mu^2\notag.
\end{align}
It follows from (\ref{equ11}) and (\ref{equ7}) that
\begin{align*}
\left\vert P-I\right\vert_{\hat D}\leq\vert P\vert_{D(s,r)}+\vert I\vert_{\hat D}\leq2\gamma_0^{n+m+2}s^{m}\mu.
\end{align*}
For $\vert p\vert=m+1$, let $\int$ be the obvious antiderivative of $\frac{\partial^{p}}{\partial y^p}$. Then the Cauchy estimate of $P-I$ on $D_\alpha$ yields
\begin{align*}
\left\vert II\right\vert_{D_\alpha}&=\left\vert\int\frac{\partial^{p}}{\partial y^p}\sum_{\vert k\vert\leq K_+,~\vert\imath\vert>m}p_{k\imath}y^{\imath} e^{\sqrt{-1}\langle k,x\rangle}dy\right\vert_{D_{\alpha}}\\
&\leq\left\vert\int\left\vert\frac{\partial^{p}}{\partial y^p}(P-I)\right\vert dy\right\vert_{D_{\alpha}}\\
&\leq\left\vert\frac{c}{s^{m+1}}\int\left\vert P-I\right\vert_{\hat{D}}dy\right\vert_{D_{\alpha}}\\
&\leq2\frac{c}{s^{m+1}}\gamma_0^{n+m+2}s^{m}\mu(\alpha s)^{m+1}\\
&\leq c\gamma_0^{n+m+2}s^{m}\mu^2.
\end{align*}
Thus,
\begin{align}\label{equ12}
\left\vert P-R\right\vert_{D_\alpha}=\left\vert I+II\right\vert_{D_\alpha}\leq c\gamma_0^{n+m+2}s^{m}\mu^2,
\end{align}
and therefore,
\begin{align}\label{equ13}
\left\vert R\right\vert_{D_\alpha}\leq\vert P-R\vert_{D_\alpha}+\vert P\vert_{D(s,r)}\leq c\gamma_0^{n+m+2}s^{m}\mu.
\end{align}

Next, we estimate $\left\|P-R\right\|_{C^\beta}$. In view of the definition of $\|\cdot\|_{C^\beta}$, for $\forall y,x\in D_\alpha$, we have
\begin{align}
\left\|P-R\right\|_{C^\beta}&=\sup_{\xi\neq\zeta}\frac{\left\vert P(x,y,\xi)-R(x,y,\xi)-(P(x,y,\zeta)-R(x,y,\zeta))\right\vert}{\left\vert\xi-\zeta\right\vert^\beta}\notag\\
&\leq\sup_{\xi\neq\zeta}\frac{\left\vert\int\right\vert\frac{\partial^p}{\partial y^p}(P(x,y,\xi)-R(x,y,\xi)-(P(x,y,\zeta)-R(x,y,\zeta)))\left\vert dy\right\vert}{\vert\xi-\zeta\vert^\beta}\notag\\\label{equ14}
&\leq\sup_{\xi\neq\zeta}\frac{\left\vert\frac{c}{s^{m+2}}\int\left\vert P(x,y,\xi)-P(x,y,\zeta)\right\vert dy\right\vert}{\vert\xi-\zeta\vert^\beta}\notag\\
&\leq\sup_{\xi\neq\zeta}\frac{c}{s^{m+1}}\frac{\vert P(x,y,\xi)-P(x,y,\zeta)\vert}{\vert\xi-\zeta\vert^\beta}(\alpha s)^{m+1}\notag\\
&\leq c\mu\|P\|_{C^\beta}\leq c\gamma_0^{n+m+2}s^m\mu^2,
\end{align}
where the third inequality follows from Cauchy estimate and the last inequality follows from (\ref{equ11}).

Similarly, we get
\begin{align}\label{equ15}
\|R\|_{C^\beta}<\|P-R\|_{C^\beta}+\|P\|_{C^\beta}\leq c\gamma_0^{n+m+2}s^m\mu.
\end{align}

It follows from (\ref{equ12}), (\ref{equ13}), (\ref{equ14}) and  (\ref{equ15}) that (\ref{equ9}) and (\ref{equ10}) hold.

The proof is complete.
\end{proof}
\subsubsection{Homological Equation}\label{sub3.2}
As usual, we shall construct a symplectic transformation as the time 1-map $\phi_{F}^1$ of the flow generated by a Hamiltonian $F$ to eliminate all resonant terms in $R$, i.e., all terms
\begin{align*}
p_{k\imath}y^{\imath}e^{\sqrt{-1}\langle k,x\rangle},~~~~0<\vert k\vert\leq K_+,\vert\imath\vert\leq m.
\end{align*}

To do so, we first construct a Hamiltonian $F$ of the form
\begin{equation}\label{eqb5}
F=\sum_{0<\vert k\vert\leq K_+,\vert\imath\vert\leq m}f_{k\imath}y^{\imath}e^{\sqrt{-1}\langle k,x\rangle},
\end{equation}
satisfying the equation
\begin{equation}\label{eqb6}
\{N,F\}+R-[R]=0,
\end{equation}
where $[R]=\frac{1}{(2\pi)^n}\int_{T^n}R(y,x)dx$ is the average of the truncation $R$.

Substituting (\ref{eqb5}) into (\ref{eqb6}) yields that
\begin{align*}
&-\sum_{0<\vert k\vert<K_+,\vert\imath\vert\leq m}\sqrt{-1}\left\langle k,\omega(\xi_0)+\partial_y\bar{h}\right\rangle f_{k\imath}y^\imath e^{\sqrt{-1}\langle k,x\rangle}\\
&+\sum_{0<\vert k\vert<K_+,\vert\imath\vert\leq m} p_{k\imath}y^\imath e^{\sqrt{-1}\langle k,x\rangle}=0.
\end{align*}
By comparing the coefficients above, we then obtain the following quasi-linear equations:
\begin{equation}\label{eqb7}
\sqrt{-1}\left\langle k,\omega(\xi_0)+\partial_y\bar{h}\right\rangle f_{k\imath}=p_{k\imath},~~~\vert\imath\vert\leq m,~~~0<\vert k\vert\leq K_+.
\end{equation}
We declare that the quasi-linear equations (\ref{eqb7}) is solvable under some suitable conditions.
The details can be seen in the following lemma:
\begin{lemma}\label{le2}
Assume that
\begin{align*}
&\textbf{\textsc{(H2)}}: max_{\vert i\vert\leq 2}\left\vert\left\| \partial_y^i\bar{h}- \partial_y^i\bar{h}_0\right\|\right\vert_{D(s)}\leq \mu_0^{\frac{1}{2}}, \\
&\textbf{\textsc{(H3)}}:2s<\frac{\gamma_0}{(M^*+2)K_+^{\tau+1}},
\end{align*}
where
\begin{align*}
M^*=\max_{\vert i\vert\leq 2, y\in D(s)}\left\vert \partial_{y}^i\bar h_0(\xi_0,y)\right\vert.
\end{align*}
Then the quasi-linear equations (\ref{eqb7}) can be uniquely solved on $D(s)$ to obtain a family of functions $f_{k\imath}$ which are analytic in $y$, and satisfy the following properties:
\begin{equation}\label{eqb8}
\left\vert\left\|\partial_y^if_{k\imath}\right\|\right\vert_{D(s)}\leq c_2|k|^{(\vert i\vert+1)\tau+\vert i\vert}\gamma_0^{n+m+1-\vert i\vert}s^{m-\vert\imath\vert}\mu e^{-\vert k\vert r},
\end{equation}
for all $\vert\imath\vert\leq m, 0<\vert k\vert\leq K_+, \vert i\vert\leq 2$, where $c_2$ is a constant.
\end{lemma}
\begin{proof}
For $\forall y\in D(s)$, by $\textbf{\text{(H2)}},\textbf{(H3)}$, we have
\begin{equation*}
\left\vert\partial_y\bar h\right\vert_{D(s)}=\left\vert(\partial_y\bar h-\partial_y\bar h_0)+\partial_y\bar h_0\right\vert_{D(s)}
\leq(1+M^*)\vert y\vert<(1+M^*)s
<\frac{\gamma_0}{2\vert k\vert^{\tau+1}}
\end{equation*}
and
\begin{equation*}
\left\|\partial_y\bar h\right\|_{C^\beta}=\sup_{\xi_*\neq\xi_{**},~\xi_*,\xi_{**}\in{O}}\frac{\left\vert\partial_y\bar h(y,\xi_*)-\partial_y\bar h(y,\xi_{**})\right\vert}{\vert\xi_*-\xi_{**}\vert^\beta}\leq\mu_0^\frac{1}{2}\vert y\vert<\vert s\vert<\frac{\gamma_0}{2\vert k\vert^{\tau+1}},
\end{equation*}
which imply that
\begin{equation}\label{eqb36}
\left\vert\left\|\partial_y\bar{h}\right\|\right\vert_{D(s)}<\frac{\gamma_0}{2\vert k\vert^{\tau+1}}.
\end{equation}
It follows from (\ref{eqb36}) and ${\rm(A2)}$ that
\begin{equation}\label{eqb37}
\left\vert\left\|\left\langle k,\omega(\xi_0)+\partial_y \bar h(y)\right\rangle\right\|\right\vert_{D(s)}>\frac{\gamma_0}{\vert k\vert^\tau}-\frac{\gamma_0}{2\vert k\vert^\tau}=\frac{\gamma_0}{2\vert k\vert^\tau}.
\end{equation}
Hence
\begin{equation}\label{eqb38}
L_k=\sqrt{-1}\left\langle k,\omega(\xi_0)+\partial_y\bar h(y)\right\rangle
\end{equation}
is invertible, and
\begin{equation}\label{fki}
f_{k\imath}=L_k^{-1} p_{k\imath},
\end{equation}
for all $y\in D(s)$, $0<\vert k\vert\leq K_+$, $\vert\imath\vert\leq m$.
Let $0<\vert k\vert\leq K_+$. We note by the first inequality of (\ref{pki}) and Cauchy estimate that
\begin{align}\label{eqb9}
\left\vert\left\|p_{k\imath}\right\|\right\vert&\leq\left\vert\left\| \partial_y^\imath P\right\|\right\vert_{\tilde{D}}e^{-\vert k\vert r}\leq\gamma_0^{n+m+2}s^{m-\vert\imath\vert}\mu e^{-\vert k\vert r},~~~~\vert\imath\vert\leq m,
\end{align}
and by (\ref{eqb37}) and (\ref{eqb38}) that
\begin{align}\label{lk-1}
\left\vert\left\|\partial_y^iL_k^{-1}\right\|\right\vert_{D(s)}&\leq c_2\left\vert k\right\vert^{\vert i\vert}\left\|\left\vert L_k^{-1}\right\|\right\vert_{D(s)}^{\vert i\vert+1}\leq c_2\frac{\vert k\vert^{(\vert i\vert+1)\tau+\vert i\vert}}{\gamma_0^{\vert i\vert+1}},~~~~\vert i\vert\leq2.
\end{align}
So, by (\ref{fki}), (\ref{eqb9}) and (\ref{lk-1}), we get
\begin{align*}
\left\vert\left\| \partial_y^if_{k\imath}\right\|\right\vert_{D(s)}&\leq c_2\frac{\vert k\vert^{(\vert i\vert+1)\tau+\vert i\vert}}{\gamma_0^{\vert i\vert+1}}\gamma_0^{n+m+2}s^{m-\vert\imath\vert}\mu e^{-\vert k\vert r}\\
&=c_2\vert k\vert^{(\vert i\vert+1)\tau+\vert i\vert}\gamma_0^{n+m+1-\vert i\vert}s^{m-\vert\imath\vert}\mu e^{-\vert k\vert r},~~~~\vert i\vert\leq 2.
\end{align*}

The proof is complete.
\end{proof}

Next, we apply the above transformation $\phi_F^1$ to Hamiltonian $H$, i.e.,
\begin{align*}
H\circ\phi_F^1=&(N+R)\circ\phi_F^1+(P-R)\circ\phi_F^1\\
=&(N+R)+\{N,F\}+\int_0^1\left\{(1-t)\{N,F\}+R,F\right\}\circ\phi_F^tdt\\
&+(P-R)\circ\phi_F^1\\
=&N+[R]+\int_0^1\left\{R_t,F\right\}\circ\phi_F^tdt+\left(P-R\right)\circ\phi_F^1\\
=&:\bar N_++\bar P_+,
\end{align*}
where
\begin{align}\label{eqb10}
& \bar N_+=N+[R]=e_++\langle\omega(\xi),y\rangle+\left\langle \sum_{j=0}^{\nu}p_{01}^j(\xi),y\right\rangle+\bar{h}_+(y,\xi),\\\label{equ5}
&e_+=e+p_{00}^\nu,\\\label{equ8}
&\bar{h}_+=\bar{h}(y,\xi)+[R]-p_{00}^\nu-\left\langle p_{01}^\nu(\xi),y\right\rangle,\\\label{eqb11}
& \bar P_+=\int_0^1\{R_t,F\}\circ\phi_F^tdt+(P-R)\circ\phi_F^1,\\
&R_t=(1-t)[R]+tR.\notag
\end{align}

\subsubsection{Translation}\label{subb3.3}
In this subsection, we will construct a translation so as to keep the frequency unchanged. It should be pointed out that we present a new way to move parameters, but in the usual KAM iteration, one has to dig out a decreasing series of parameter domains, in which the Diophantine condition doesn't hold, see \cite{chow,han,li1,poschel1,poschel,qian2,qian1}.

Consider the translation
$$\phi:x\rightarrow x,~~~~~y\rightarrow y,~~~~~\tilde{\xi}\rightarrow\tilde{\xi}+\xi_+-\xi,$$
where $\xi_+$ is to be determined.
Let
$$\Phi_+=\phi_F^1\circ\phi.$$
Then
\begin{align}
H\circ\Phi_+&=N_++P_+,\notag\\\label{eqb12}
N_+&=\bar N_+\circ\phi=e_++\left\langle\omega(\xi_{+}),y\right\rangle+\left\langle \sum_{j=0}^{\nu}p_{01}^j(\xi_+),y\right\rangle+\bar{h}_+(y,\xi_+)
,\\\label{eqb13}
P_+&=\bar P_+\circ\phi.
\end{align}

\subsubsection{Frequency-preserving}\label{sub3.3}
In this subsection, we will show that the frequency can be preserved in the iteration process.
Recall the topological degree condition $\rm(A0)$ and the weak convexity condition $\rm(A1)$. The former ensures that the parameter $\xi_+$ can be found in the parameter set to keep the frequency unchanged at this KAM step. The later assures that the distance between $\xi_+$ and $\xi$ is smaller than the distance between $\xi$ and $\xi_{\nu-1}$, i.e., the sequence of parameters is convergent after infinite steps of iteration. The following lemma is crucial to our arguments.
\begin{lemma}\label{leb3}
Assume that
\begin{align*}
\textbf{(\textsc{H4})}:\left\vert\left\|\sum_{j=0}^{\nu}p_{01}^j\right\|\right\vert_{D(s,r)}<\mu_0^{\frac{1}{2}}.
\end{align*}
There exists
$\xi_+\in B_{c\mu^{{1}/{L}}}(\xi)\subset O^o$
such that
\begin{align}\label{eqb14}
\omega(\xi_+)+\sum_{j=0}^{\nu}p_{01}^j(\xi_+)=\omega(\xi_0).
\end{align}
\end{lemma}
\begin{proof}
The proof will be completed by an induction on $\nu$. We start with the case $\nu=0$. It is obvious that $\omega(\xi_0)=\omega(\xi_0)$. Now assume that for some $\nu>0$ we have got
\begin{align}\label{equ1}
\omega(\xi_i)+\sum_{j=0}^{i-1}p_{01}^j(\xi_i)=\omega(\xi_0),~~\xi_i\in B_{c\mu_{i-1}^{{1}/{L}}}(\xi_{i-1})\subset O^o,~~~i=1,\cdots,\nu.
\end{align}
We need to find $\xi_+$ near $\xi$ such that
\begin{align}\label{equ2}
\omega(\xi_+)+\sum_{j=0}^{\nu}p_{01}^j(\xi_+)=\omega(\xi_0).
\end{align}
In view of the property of topological degree, $\textbf{\textsc{(H4)}}$ and $\rm(A0)$, we have
\begin{align*}
\deg\left(\omega(\cdot)+\sum_{j=0}^{\nu}p_{01}^j(\cdot),O^o,\omega(\xi_0)\right)=\deg\left(\omega(\cdot),O^o,\omega(\xi_0)\right)\neq0,
\end{align*}
i.e., there exists at least a $\xi_+\in O^o$ such that (\ref{eqb14}) holds.

Next, we estimate $\vert\xi_+-\xi\vert$. (\ref{equ10}) in Lemma \ref{leb1} implies that
\begin{align*}
\left\|p_{01}^j\right\|_{C^\beta}<c\mu_j,~~~j=0,1,\cdots,\nu,
\end{align*}
i.e.,
\begin{align}\label{equ16}
\left\vert p_{01}^j(\xi_+)-p_{01}^j(\xi)\right\vert<c\mu_j\left\vert\xi_+-\xi\right\vert^\beta,~~\forall\xi_+, \xi\in{O}.
\end{align}
According to (\ref{equ1}) and (\ref{equ2}), we get
\begin{align}\label{equ3}
\omega(\xi_+)-\omega(\xi)+\sum_{j=0}^{\nu-1}\left(p_{01}^{j}(\xi_+)-p_{01}^j{(\xi)}\right)=-p_{01}^{\nu}(\xi_+).
\end{align}
This together with $\rm(A1)$ and (\ref{equ16}) yields
\begin{align}
\left\vert p_{01}^\nu(\xi_+)\right\vert&=\left\vert\omega(\xi_+)-\omega(\xi)+\sum_{j=0}^{\nu-1}(p_{01}^{j}(\xi_+)-p_{01}^j{(\xi)})\right\vert\notag\\
&\geq\left\vert\omega(\xi_+)-\omega(\xi)\right\vert-\sum_{j=0}^{\nu-1}\left\vert p_{01}^{j}(\xi_+)-p_{01}^j{(\xi)}\right\vert\notag\\
&\geq\sigma\left\vert\xi_+-\xi\right\vert^L-c\sum_{j=0}^{\nu-1}\mu_j\left\vert\xi_+-\xi\right\vert^\beta\notag\\\label{equ19}
&\geq(\sigma-c\sum_{j=0}^{\nu-1}\mu_j)\left\vert\xi_+-\xi\right\vert^L,
\end{align}
where the last inequality follows from $0<L\leq\beta$ in $\rm(A1)$. Then, by (\ref{equ19}) and (\ref{equ10}) in Lemma \ref{leb1}, we have
\begin{align*}
\vert\xi_+-\xi\vert^L&\leq\frac{p_{01}^\nu(\xi_+)}{\sigma-c\sum_{j=0}^{\nu-1}\mu_j}<\frac{c\mu}{\sigma-c\sum_{j=0}^{\nu-1}\mu_j}<\frac{2c\mu}{\sigma},
\end{align*}
which implies $\xi_+\in B_{c\mu^{{1}/{L}}}(\xi)$. From $\xi\in O^o$ in (\ref{equ1}) and the fact $\varepsilon$ is small enough (i.e., $\mu$ is small enough), we have $B_{c\mu^{{1}/{L}}}(\xi)\subset O^o$.

The proof is complete.
\end{proof}

\subsubsection{Estimate on $N_+$}
Now, we give the estimate on the next step $N_+$.
\begin{lemma}\label{leb4}
There is a constant $c_3$ such that the following holds:
\begin{align}\label{eqb15}
&\left\vert\xi_{+}-\xi\right\vert\leq c_3\mu^\frac{1}{L},\\\label{eqb16}
&\left\vert e_+-e\right\vert\leq c_3s^m\mu,\\\label{equ4}
&\left\vert\left\|\bar h_+-\bar h\right\|\right\vert_{D(s)}\leq
c_3s^m\mu.
\end{align}
\end{lemma}
\begin{proof}
It is obvious by $\xi_+\in B_{c\mu^{{1}/{L}}}(\xi)$ in Lemma \ref{leb3} that (\ref{eqb15}) holds. It follows from (\ref{equ5}) and (\ref{equ8}) that (\ref{eqb16}) and (\ref{equ4}) hold.
\end{proof}

\subsubsection{Estimate on $\Phi_+$}
Recall that $F$ is as in (\ref{eqb5}) with the coefficients and its estimate given by Lemma \ref{le2}. Then, we have the following estimate on $F$.
\begin{lemma}\label{leb5}
There is a constant $c_4$ such that for all $\vert j\vert+\vert i\vert\leq 2$,
\begin{align}\label{equ17}
&\left\vert\left\| \partial_x^j\partial_y^iF\right\|\right\vert_{\hat D}\leq c_4\gamma_0^{n+m+1}s^{m-\vert i \vert}\mu\Gamma(r-r_+).
\end{align}
\end{lemma}

\begin{proof}
By (\ref{eqb5}) and (\ref{eqb8}), we have
\begin{align*}
\left\vert\left\| \partial_x^j\partial_y^iF\right\|\right\vert_{\hat D}&\leq \sum_{\vert\imath\vert\leq m,0<\vert k\vert\leq K_+}\vert k\vert^j\left\vert\left\|\partial_y^i(f_{k\imath}y^\imath)\right\|\right\vert_{D(s)}e^{\vert k\vert(r_++\frac{7}{8}(r-r_+))}\\
&\leq c_4\sum_{0<\vert k\vert\leq K_+}\vert k\vert^{\tau+\vert j\vert}\gamma_0^{n+m+1}s^{m-\vert i\vert}\mu e^{-\vert k\vert\frac{r-r_+}{8}}\\
&\leq c_4\gamma_0^{n+m+1}s^{m-\vert i\vert}\mu\Gamma(r-r_+).
\end{align*}

The proof is complete.
\end{proof}

\begin{lemma}\label{leb6}
Assume that
\begin{align*}
&\textbf{\textsc{(H5)}}:c_4s^{m-1}\mu\Gamma(r-r_+)<\frac{1}{8}(r-r_+),\\
&\textbf{\textsc{(H6)}}:c_4s^m\mu\Gamma(r-r_+)<\frac{1}{8}\alpha s.
\end{align*}
Then the following holds.
\begin{description}
\item[\rm(1)]For all $0\leq t\leq 1$, the mappings
\begin{align}\label{eqb17}
\phi_F^t&:D_{\frac{1}{4}\alpha}\rightarrow D_{\frac{1}{2}\alpha},\\\label{eqb18}
\phi&:O\rightarrow O_+,
\end{align}
are well defined.
\item[\rm(2)]$\Phi_+:D_+\rightarrow D(s,r).$
\item[\rm(3)]There is a constant $c_5$ such that
\begin{align*}
\left\vert\left\|\phi_F^t-id\right\|\right\vert_{\tilde{D}}&\leq c_5\mu\Gamma(r-r_+),\\
\left\vert\left\|D\phi_F^t-Id\right\|\right\vert_{\tilde{D}}&\leq c_5\mu\Gamma(r-r_+),\\
\left\vert\left\|D^2\phi_F^t\right\|\right\vert_{\tilde{D}}&\leq c_5\mu\Gamma(r-r_+).
\end{align*}
\item[\rm(4)]
\begin{align*}
\left\vert\left\|\Phi_+-id\right\|\right\vert_{\tilde{D}}&\leq c_5\mu\Gamma(r-r_+),\\
\left\vert\left\|D\Phi_+-Id\right\|\right\vert_{\tilde{D}}&\leq c_5\mu\Gamma(r-r_+),\\
\left\vert\left\|D^2\Phi_+\right\|\right\vert_{\tilde{D}}&\leq c_5\mu\Gamma(r-r_+).
\end{align*}
\end{description}
\end{lemma}
\begin{proof}
(1)~~(\ref{eqb18}) immediately  follows from (\ref{eqb15}) and the definition of $O_+$.
To verify (\ref{eqb17}), we denote $\phi_{F_1}^t$, $\phi_{F_2}^t$ as the components of $\phi_{F}^t$ in the $y$, $x$ planes, respectively. Let $X_F=(F_y,-F_x)^\top$ be the vector field generated by $F$. Then
\begin{align}\label{eqb19}
\phi_F^t=id+\int_0^tX_F\circ \phi_F^udu,~~~~0\leq t\leq 1.
\end{align}
For any $(y,x)\in D_{\frac{1}{4}\alpha}$, we let $t_*=\sup\{t\in[0,1]:\phi_F^t(y,x)\in D_\alpha\}$. Then for any $0\leq t\leq t_*$, by $(y,x)\in D_{\frac{1}{4}\alpha}$, (\ref{equ17}) in Lemma \ref{leb5}, $\textbf{\textsc{(H5)}}$ and \textbf{\textsc{(H6)}}, we can get the following estimates:
\begin{align*}
\left\vert\left\|\phi_{F_1}^t(y,x)\right\|\right\vert_{D_{\frac{1}{4}\alpha}}&\leq\left\vert y\right\vert+\int_0^t\left\vert\left\|F_x\circ\phi_F^u\right\|\right\vert_{D_{\frac{1}{4}\alpha}}du\\
&\leq \frac{1}{4}\alpha s+c_4s^{m}\mu\Gamma(r-r_+)\\
&<\frac{3}{8}\alpha s,\\
\left\vert\left\|\phi_{F_2}^t(y,x)\right\|\right\vert_{D_{\frac{1}{4}\alpha}}&\leq\vert x\vert+\int_0^t\left\vert\left\|F_y\circ\phi_F^u\right\|\right\vert_{D_{\frac{1}{4}\alpha}}du\\
&\leq r_++\frac{1}{8}(r-r_+)+c_4s^{m-1}\mu\Gamma(r-r_+)\\
&<r_++\frac{2}{8}(r-r_+).
\end{align*}
Thus, $\phi_F^t\in D_{\frac{1}{2}\alpha}\subset D_\alpha$, i.e. $t_*=1$ and (1) holds.

(2) It follows from (1) that (2) holds.

(3) Using (\ref{equ17}) in Lemma \ref{leb5} and (\ref{eqb19}), we immediately have
\begin{align*}
\left\vert\left\|\phi_F^t-id\right\|\right\vert_{\tilde{D}}\leq c_5\mu\Gamma(r-r_+).
\end{align*}
By (\ref{equ17}) in Lemma \ref{leb5}, (\ref{eqb19}) and Gronwall Inequality, we get
\begin{align*}
\left\vert\left\|D\phi_F^t-Id\right\|\right\vert_{\tilde{D}}&\leq\left\vert\left\|\int_0^tDX_F\circ \phi_F^\lambda D\phi_F^\lambda d\lambda\right\|\right\vert_{\tilde{D}}\\
&\leq\int_0^t\left\vert\left\|DX_F\circ \phi_F^\lambda\right\|\right\vert_{\tilde{D}}\left\vert\left\|D\phi_F^\lambda-Id\right\|\right\vert_{\tilde{D}}d\lambda+\int_0^t\left\vert\left\|DX_F\circ \phi_F^\lambda\right\|\right\vert_{\tilde{D}}d\lambda\\
&\leq c_5\mu\Gamma(r-r_+).
\end{align*}
It follows from the induction and a similar argument that we have the estimates on the 2-order derivatives of $\phi_F^t$, i.e.,
\begin{align*}
\left\vert\left\|D^2\phi_F^t\right\|\right\vert_{\tilde{D}}\leq c_5\mu\Gamma(r-r_+).
\end{align*}

(4) now follows from (3).

The proof is complete.
\end{proof}

\subsubsection{Estimate on $P_+$}
\indent In the following, we estimate the next step $P_+$.
\begin{lemma}\label{leb7}
Assume $\textbf{\textsc{(H1)}}$-$\textbf{\textsc{(H6)}}$. Then there is a constant $c_6$ such that,
\begin{equation}\label{eqb20}
\left\vert\left\|P_+\right\|\right\vert_{D_+}\leq c_6\gamma_0^{n+m+2}s^m\mu^2(\Gamma^2(r-r_+)+\Gamma(r-r_+)).
\end{equation}
Moreover, if
\begin{align*}
\textbf{\textsc{(H7)}}:\mu^\rho(\Gamma^2(r-r_+)+\Gamma(r-r_+))\leq1
\end{align*}
then,
\begin{equation}\label{equ20}
\left\vert\left\|P_+\right\|\right\vert_{D_+}\leq \gamma_0^{n+m+2}s_+^m\mu_+.
\end{equation}
\end{lemma}
\begin{proof}
By (\ref{equ9}) and (\ref{equ10}) in Lemma \ref{leb1}, (\ref{equ17}) in Lemma \ref{leb5} and Lemma \ref{leb6} (3), we have that, for all $0\leq t\leq 1$,
\begin{align*}
\left\vert\left\| \{R_t,F\}\circ\phi_F^t\right\|\right\vert_{D_{\frac{1}{4}\alpha}}&\leq c\gamma_0^{n+m+2}s^m\mu^2\Gamma^2(r-r_+),\\
\left\vert\left\| (P-R)\circ\phi_F^1\right\|\right\vert_{D_{\frac{1}{4}\alpha}}&\leq c\gamma_0^{n+m+2}s^m\mu^2\Gamma(r-r_+).
\end{align*}
So, by (\ref{eqb11}),
\begin{equation*}
\left\vert\left\|\bar P_+\right\|\right\vert_{D_{\frac{1}{4}\alpha}}\leq c\gamma_0^{n+m+2}s^m\mu^2\left(\Gamma^2(r-r_+)+\Gamma(r-r_+)\right).
\end{equation*}
By \textbf{(H7)}, we see that,
\begin{align*}
\left\vert\left\|P_+\right\|\right\vert_{D_+}&\leq 8^mc_0\mu^{1+\rho}s_+^m\mu^{1-2\rho-\frac{m}{m+1}}\gamma_0^{n+m+2}\left(\mu^\rho\left(\Gamma^2(r-r_+)+\Gamma(r-r_+)\right)\right)\\
&\leq\gamma_0^{n+m+2}s_+^m\mu_+,
\end{align*}
which implies (\ref{equ20}).

The proof is complete.
\end{proof}

This completes one cycle of KAM steps.

\section{Proof of Theorem \ref{th1}}\label{sec:b4}

\subsection{Iteration lemma}
In this subsection, we will prove an iteration lemma which guarantees the inductive construction of the transformations in all KAM steps.

Let $r_0,s_0,\gamma_0,
\mu_0,H_0,N_0,e_0,\bar h_0,P_0$ be given at the beginning of Section \ref{sec:b3} and let
$D_0=D(s_0,r_0)$, $K_0=0$, $\Phi_0=id$. We define the following sequence inductively for all $\nu=1,2,\cdots$.
\begin{align*}
r_\nu&=r_0\left(1-\sum_{i=1}^\nu\frac{1}{2^{i+1}}\right),\\
s_\nu&=\frac{1}{8}\alpha_{\nu-1}s_{\nu-1},\\
\alpha_\nu&=\mu_\nu^{2\rho}=\mu_\nu^{\frac{1}{m+1}},\\
\mu_\nu&=8^mc_0\mu_{\nu-1}^{1+\rho},\\
K_\nu&=\left(\left[\log\left(\frac{1}{\mu_{\nu-1}}\right)\right]+1\right)^{3\eta},\\
\tilde{D}_\nu&=D\left(\frac{1}{2} s_\nu, r_\nu+\frac{6}{8}\left(r_{\nu-1}-r_\nu\right)\right).
\end{align*}
\begin{lemma}\label{leb8}
Denote
\begin{align*}
\mu_*=\frac{\mu_0}{\left(M^*+2\right)^{m-1}K_1^{(\tau+1)(m-1)}}.
\end{align*}
If $\varepsilon$ is small enough, then the KAM step described on the above is valid for all $\nu=0,1,\cdots$, resulting the sequences
$$H_\nu, N_\nu, e_\nu, \bar h_\nu, P_\nu, \Phi_\nu,$$
$\nu=1,2,\cdots,$ with the following properties:
\begin{description}
\item[\rm(1)]
\begin{align}\label{eq31}
\left\vert e_{\nu+1}-e_{\nu}\right\vert&\leq\frac{\mu_*^{\frac{1}{2}}}{2^{\nu}},\\\label{eq32}
\left\vert e_\nu-e_{0}\right\vert&\leq2\mu_*^{\frac{1}{2}},\\\label{eq33}
\left\vert\left\|\bar h_{\nu+1}-\bar h_{\nu}\right\|\right\vert_{D(s_{\nu})}&\leq\frac{\mu_*^{\frac{1}{2}}}{2^{\nu}},\\\label{eq34}
\left\vert\left\|\bar h_\nu-\bar h_{0}\right\|\right\vert_{D(s_\nu)}&\leq2\mu_*^{\frac{1}{2}},\\\label{eq35}
\left\vert\left\|P_\nu\right\|\right\vert_{D(s_\nu,r_\nu)}&\leq\frac{\mu_*^{\frac{1}{2}}}{2^{\nu}},\\\label{eq55}
\left\vert\xi_{\nu+1}-\xi_{\nu}\right\vert&\leq(\frac{\mu_*^{\frac{1}{2}}}{2^{\nu}})^{\frac{1}{L}}.
\end{align}
\item[\rm(2)] $\Phi_{\nu+1}:\tilde{D}_{\nu+1}\rightarrow \tilde{D}_{\nu}$ is symplectic, and
\begin{align}\label{eq39}
\left\vert\left\|\Phi_{\nu+1}-id\right\|\right\vert_{\tilde{D}_{\nu+1}}\leq\frac{\mu_*^{\frac{1}{2}}}{2^{\nu}}.
\end{align}
 Moreover, on $D_{\nu+1}$,
\begin{equation*}
H_{\nu+1}=H_\nu\circ\Phi_{\nu+1}=N_{\nu+1}+P_{\nu+1}.
\end{equation*}
\end{description}
\end{lemma}
\begin{proof}
The proof amounts to the verification of $\textbf{(H1)}$-$\textbf{(H7)}$ for all $\nu$. For simplicity, we let $r_0=1$.
It follows from $\varepsilon$ small enough that $\mu_0$ is small. So, we see that $\textbf{(H2)}$, $\textbf{(H4)}$-$\textbf{(H7)}$ hold for $\nu=0$.
From (\ref{eqb2}), $\textbf{(H3)}$ holds for $\nu=0$.
According to the definition of $\mu_\nu$, we see that
\begin{equation}\label{eq44}
\mu_\nu=(8^mc_0)^{\frac{(1+\rho)^\nu-1}{\rho}}\mu_0^{(1+\rho)^\nu}.
\end{equation}
Let $\zeta\gg1$ be fixed and $\mu_0$ be small enough so that
\begin{equation}\label{eq29}
\mu_0<(\frac{1}{8^mc_0\zeta})^{\frac{1}{\rho}}<1.
\end{equation}
Then
\begin{align}
\mu_1&=8^mc_0\mu_0^{1+\rho}<\frac{1}{\zeta}\mu_0<1,\notag\\
\mu_2&=8^mc_0\mu_1^{1+\rho}<\frac{1}{\zeta}\mu_1<\frac{1}{\zeta^2}\mu_0,\notag\\
\vdots\notag\\\label{eq28}
\mu_\nu&=8^mc_0\mu_{\nu-1}^{1+\rho}<\cdots<\frac{1}{\zeta^\nu}\mu_0.
\end{align}
Denote
\begin{equation*}
\Gamma_\nu=\Gamma(r_\nu-r_{\nu+1}).
\end{equation*}
We notice that
\begin{equation}\label{eq26}
\frac{r_\nu-r_{\nu+1}}{r_0}=\frac{1}{2^{\nu+2}}.
\end{equation}
Since
\begin{align*}
\Gamma_\nu&\leq\int_1^\infty t^{3\tau+5}e^{-\frac{t}{2^{\nu+5}}}\\
&\leq(3\tau+5)!2^{(\nu+5)(3\tau+5)},
\end{align*}
it is obvious that if $\zeta$ is large enough, then
\begin{align*}
\mu_\nu^\rho\Gamma_\nu^i<\mu_\nu^\rho(\Gamma_\nu^i+\Gamma_\nu)\leq1,~~~i=1,2,
\end{align*}
which implies that $\textbf{(H7)}$ holds for all $\nu\geq1$, and
\begin{equation}\label{eq27}
\mu_\nu\Gamma_\nu\leq\mu_\nu^{1-\rho}\leq\frac{\mu_0^{1-\rho}}{\zeta^{(1-\rho)\nu}}.
\end{equation}
By (\ref{eq26}) and (\ref{eq27}), it is easy to verify that $\textbf{(H5)}$, $\textbf{(H6)}$ hold for all $\nu\geq1$ as $\zeta$ is large enough and $\varepsilon$ is small enough.

By (\ref{equ10}) in Lemma \ref{leb1} and (\ref{eq28}), we have
\begin{align*}
\left\vert\left\|\sum_{j=0}^{\nu}p_{01}^j\right\|\right\vert_{D(s_\nu,r_\nu)}<c\sum_{j=0}^{\nu}\mu_j<c\sum_{j=0}^{\nu}\frac{1}{\zeta^j}\mu_0<c\mu_0^{\frac{1}{2}},
\end{align*}
which implies $\textbf{(H4)}$.

To verify $\textbf{(H3)}$, we observe by (\ref{eq44}) and (\ref{eq28}) that
\begin{equation*}
\frac{1}{4}\left(M^*+2\right)\mu_{\nu-1}^{2\rho}K_{\nu+1}^{\tau+1}<\frac{1}{2^{\nu+2}},
\end{equation*}
as $\zeta$ is large enough.
Then
\begin{align}
2\left(M^*+2\right)s_\nu K_{\nu+1}^{\tau+1}&\leq\frac{s_{\nu-1}}{4}\left(M^*+1\right)\mu_{\nu-1}^{2\rho}K_{\nu+1}^{\tau+1}\\
&\leq\frac{s_0}{2^{\nu+2}}<\frac{\gamma_0}{2^{\nu+2}}<\gamma_0,
\end{align}
which verifies $\textbf{(H3)}$ for all $\nu\geq1$.

Let $\zeta^{1-\rho}\geq2$ in (\ref{eq29}), (\ref{eq28}). We have that for all $\nu\geq1$
\begin{align}\label{eq30}
c_0\mu_\nu&\leq\frac{\mu_0}{2^\nu}\leq\frac{\mu_*^\frac{1}{2}}{2^\nu},\\\label{eq36}
c_0\mu_\nu\Gamma_\nu&\leq\frac{\mu_0^{1-\rho}}{2^\nu}\leq\frac{\mu_*^\frac{1}{2}}{2^\nu},\\\label{eq37}
c_0s_\nu^{m-1}\mu_\nu&\leq\frac{\mu_0^{1+2\rho(m-1)}s_0^{m-1}}{2^{\nu+3}}\leq\frac{\mu_*}{2^\nu}.
\end{align}
The verification of $\textbf{(H2)}$ follows from $(\ref{eq30})$ and an induction application of $(\ref{equ4})$ in Lemma \ref{leb4} for all $\nu=0,1,\cdots.$

Since $\left(1+\rho\right)^\eta>2$, we have
\begin{align*}
\frac{1}{2^{\nu+6}}\left(\left[\log\frac{1}{\mu}\right]+1\right)^\eta&\geq\frac{1}{2^{\nu+6}}\left(\left(1-\left(1+\rho\right)^\nu\right)\log\left(8^mc_0\right)-\left(1+\rho\right)^\nu\log\mu_0\right)^\eta\\
&\geq-\frac{1}{2^{\nu+6}}\left(1+\rho\right)^{\eta\nu}\left(\log\mu_0\right)^\eta\geq1.
\end{align*}
It follows from the above that
\begin{align*}
&\log\left(n+1\right)!+\left(\nu+6\right)n\log2+3n\eta\log\left(\left[\log\frac{1}{\mu}\right]+1\right)-\frac{1}{2^{\left(\nu+6\right)}}\left(\left[\log\frac{1}{\mu}\right]+1\right)^{3\eta}\\
&\leq\log\left(n+1\right)!+\left(\nu+6\right)n\log2+3n\eta\log\left(\log\frac{1}{\mu}+2\right)-\left(\log\frac{1}{\mu}\right)^{2\eta}\\
&\leq-\log\frac{1}{\mu},
\end{align*}
as $\mu$ is small, which is ensured by making $\varepsilon$ small.
Thus,
\begin{equation*}
\int_{K_{\nu+1}}^\infty t^{n}e^{-\frac{t}{2^{\nu+6}}}dt\leq\left(n+1\right)!2^{\left(\nu+6\right)n}K_{\nu+1}^{n}e^{-\frac{K_{\nu+1}}{2^{\nu+6}}}\leq \mu,
\end{equation*}
i.e. $\textbf{(H1)}$ holds.

Above all, the KAM steps described in Section \ref{sec:b3} are valid for all $\nu$, which give the desired sequences stated in the lemma.

Now, (\ref{eq31}) and (\ref{eq33}) follow from Lemma \ref{leb4}, (\ref{eq30}) and (\ref{eq37}); by adding up (\ref{eq31}), (\ref{eq33}) for all $\nu=0,1,\cdots$, we can get (\ref{eq32}), (\ref{eq34}); (\ref{eq35}) follows from (\ref{equ20}) in Lemma \ref{leb7} and (\ref{eq30}); (\ref{eq55}) follows from (\ref{eqb15}) in Lemma \ref{leb4} and (\ref{eq30}); $(2)$ follows from Lemma \ref{leb6}.

The proof is complete.
\end{proof}
\subsection{Convergence}\label{sub2}
The convergence is standard. For the sake of completeness, we briefly give the framework of proof.
Let
\begin{align*}
\Psi^\nu=\Phi_1\circ\Phi_2\circ\cdots\circ\Phi_\nu,~~~~\nu=1,2,\cdots.
\end{align*}
By Lemma \ref{leb8}, we have
\begin{align*}
D_{\nu+1}&\subset D_\nu,\\
\Psi^\nu&:\tilde{D}_\nu\rightarrow \tilde{D}_0,\\
H_0\circ\Psi^\nu&=H=N_\nu+P_\nu,\\
N_\nu&=e_\nu+\left\langle\omega\left(\xi_\nu\right)+\sum_{j=0}^{\nu-1}p_{01}^j\left(\xi_\nu\right),y\right\rangle+\bar h_\nu(y,\xi_\nu),
\end{align*}
$\nu=0,1,\cdots,$ where $\Psi_0=id$.
Using (\ref{eq39}) and the identity
\begin{align*}
\Psi^\nu=id+\sum_{i=1}^\nu\left(\Psi^i-\Psi^{i-1}\right),
\end{align*}
it is easy to verify that $\Psi^\nu$ is uniformly convergent and denote the limitation by $\Psi^\infty$.

In view of Lemma \ref{leb8}, it is obvious to see that $e_\nu$, $\bar{h}_\nu$, $\xi_\nu$ converge uniformly about $\nu$, and denote its limitation by $e_\infty$, $\bar{h}_\infty$, $\xi_\infty$. By Lemma \ref{leb3}, we have
\begin{align}
  &\omega\left(\xi_1\right)+p_{01}^0\left(\xi_1\right)=\omega\left(\xi_0\right),\notag\\
  &\omega\left(\xi_2\right)+p_{01}^0\left(\xi_2\right)+p_{01}^1\left(\xi_2\right)=\omega\left(\xi_0\right),\notag\\
  &~~~\vdots\notag\\\label{equ6}
  &\omega\left(\xi_\nu\right)+p_{01}^0\left(\xi_\nu\right)+\cdots+p_{01}^{\nu-1}\left(\xi_\nu\right)=\omega\left(\xi_0\right).
\end{align}
Taking limits at both sides of (\ref{equ6}), we get $$\omega\left(\xi_\infty\right)+\sum_{j=0}^{\infty}p_{01}^j\left(\xi_\infty\right)=\omega(\xi_0).$$

Then, on $D(\frac{s_0}{2})$, $N_\nu$ converge uniformly to
\begin{align*}
N_\infty
&=e_\infty+\left\langle\omega(\xi_0),y\right\rangle+\bar{h}_\infty\left(y,\xi_\infty\right).
\underline{}\end{align*}
Hence, on $D\left(\frac{s_0}{2},\frac{r_0}{2}\right)$,
\begin{align*}
P_\nu=H_0\circ\Psi^\nu-N_\nu
\end{align*}
converge uniformly to
\begin{equation*}
P_\infty=H_0\circ\Psi^\infty-N_\infty.
\end{equation*}
Since
\begin{align*}
\left\vert\left\|P_\nu\right\|\right\vert_{D_\nu}\leq\gamma_0^{n+m+2}s_\nu^m\mu_\nu,
\end{align*}
by (\ref{eq30}), we have that it converges to $0$ as $\nu\rightarrow\infty$.
So, on $D(0,\frac{r_0}{2})$,
\begin{align*}
J\nabla P_\infty=0.
\end{align*}
Thus, for the given $\xi_0\in O$, the Hamiltonian
\begin{align*}
H_\infty=N_\infty+ P_\infty
\end{align*}
admits an analytic, quasi-periodic, invariant $n$-torus $\mathbb{T}^n\times\{0\}$ with the Diophantine frequency $\omega(\xi_0)$, which is the corresponding unperturbed toral frequency.

\section{Proof of Theorems \ref{th2} and \ref{th3}}\label{sec:b5}
First, we briefly give the proof framework of Theorem \ref{th2} because it can follow the KAM step in Section \ref{sec:b3}, where we mainly point out the two major differences from the proof of Theorem \ref{th1}. The first one is that the homotopy invariance and excision of topological degree are used to keep the frequency unchanged in the iteration process not by picking parameters because we consider a Hamiltonian not a family of Hamiltonian. The other one is that the transformation defines on a smaller domain because we see the action-variable as parameter and the translation of parameter is equivalent to the action-variable's.

\subsection{Proof of Theorem \ref{th2}}
In this section, we will describe the translation of action variable and state how the frequency can be preserved in the iterative process, which are different from subsection \ref{subb3.3} and subsection \ref{sub3.3}.

Let $\xi_0\in{G}$ be fixed as statement $\rm(A0)$. The Taylor expansion of Hamiltonian (\ref{eq1}) about $\xi_0$ reads
\begin{equation*}
H(y,\xi_0,x)=e_0+\left\langle\omega_0(\xi_0),y-\xi_0\right\rangle+\bar{h}(y-\xi_0)+\varepsilon P(y,x,\varepsilon),
\end{equation*}
where $e_0=h(\xi_0)$, $\omega_0(\xi_0)=\nabla h(\xi_0)$,  $\bar{h}(y-\xi_0)=O(\left\vert y-\xi_0\right\vert^2)$.
Using the transformation $(y-\xi_0)\rightarrow y$ in the above, we have
\begin{equation}\label{eq2}
H(y,\xi_0,x)=e_0+\langle\omega_0(\xi_0),y\rangle+\bar{h}(y,\xi_0)+\varepsilon P(y,x,\xi_0,\varepsilon),
\end{equation}
where $(y,x)$ lies in a complex neighborhood $D(s,r)$.
Denote
\begin{align*}
H_0&=: H(y,x,\xi_0)=N_0+P_0,\\
N_0&=:e_0+\langle\omega_0(\xi_0),y\rangle+\bar{h}_0,\\
P_0&=:\varepsilon P(y,x,\xi_0,\varepsilon).
\end{align*}
Now, suppose that at $\nu$-th step, we have arrived at the following real analytic Hamiltonian:
\begin{equation}\label{eq3}
\begin{aligned}
H&=N+P,~~~~~~~~~~~~~~~~~~~~~~~~~~~~~~~~~~~~\\
N&=e+\langle\omega(\xi),y\rangle+\bar{h}(y,\xi).
\end{aligned}
\end{equation}

Next, we will construct a translation so as to keep the frequency unchanged.
Consider the translation
$$\phi:x\rightarrow x,~~~~~y\rightarrow y+\xi_+-\xi,$$
where $\xi_+$ is to be determined.
Let
$$\Phi_+=\phi_F^1\circ\phi.$$
Then
\begin{align}
H\circ\Phi_+&=N_++P_+,\notag\\\label{eqb39}
N_+&=\bar N_+\circ\phi=e_++\langle\omega_+(\xi_+),y\rangle+\bar h_+(y),\\\label{eqb40}
P_+&=\bar P_+\circ\phi,
\end{align}
where
\begin{align}\label{eqb41}
e_+&=e+\langle\omega(\xi),\xi_+-\xi\rangle+\bar h(\xi_+-\xi)+[R](\xi_+-\xi),\\\label{eqb42}
\omega_+&=\omega(\xi)+\nabla \bar h(\xi_+-\xi)+\nabla[R](\xi_+-\xi),\\\label{eqb43}
\bar h_+&=\bar h(y+\xi_+-\xi)-\bar h(\xi_+-\xi)-\left\langle\nabla \bar h(\xi_+-\xi),y\right\rangle
+[R](y+\xi_+-\xi)\notag\\
&~~~~-[R](\xi_+-\xi)-\left\langle\nabla[R](\xi_+-\xi),y\right\rangle.
\end{align}
As in subsection \ref{sub3.3}, we will show that the frequency can be preserved in the iteration process. The following lemma is crucial to our arguments.
\begin{lemma}\label{lec5}
There exists
$\xi_+\in B_{s\mu^{{1}/{L}}}(\xi)$
such that
\begin{align}\label{equ22}
\omega_+(\xi_+)=\omega(\xi)=\cdots=\omega_0(\xi_0).
\end{align}
\end{lemma}
\begin{proof}
The proof will be completed by an induction on $\nu$. We begin with the case $\nu=0$.
It is obvious that $\omega_0(\xi_0)=\omega_0(\xi_0)$.
Now suppose that for some $\nu\geq0$ we have got
\begin{align*}
\omega_i(\xi_i)&=\omega_{i-1}(\xi_{i-1}),~~~~~\xi_i\in B_{s_{i-1}\mu_{i-1}^{{1}/{L}}}\left(\xi_{i-1}\right),
\end{align*}
where $i=1,2,\cdots,\nu.$
Then, we need to find $\xi_+$ near $\xi$ such that $\omega_+(\xi_+)=\omega(\xi)$.
In view of (\ref{eqb42}), we observe that
\begin{align}\label{eqb44}
\left\vert\omega_+(y)-\omega(y)\right\vert=O(\mu).
\end{align}
We split
\begin{align}\label{equ21}
\omega_+(y)-\omega(\xi)&=(\omega(y)-\omega(\xi))
+(\omega_+(y)-\omega(y)).
\end{align}

Consider homotopy $H_t(y):[0,1]\times O\rightarrow \mathbb{R}^n$,
\begin{align*}
H_t(y)&=:\left(\omega(y)-\omega(\xi)\right)
+t\left(\omega_+(y)-\omega(y)\right).
\end{align*}
For any $y\in\partial O$, $t\in[0,1]$,  by $\rm(A1)$, we have that
\begin{align*}
\left\vert H_t(y)\right\vert&\geq\vert\omega(y)-\omega(\xi)\vert
-\vert\omega_+(y)-\omega(y)\vert\geq\vert\omega_0(y)-\omega_{0}(\xi_0)\vert-\sum_{i=0}^\nu\left\vert\omega_{i+1}(y)-\omega_i(y)\right\vert\\
&\geq\sigma\vert y-\xi_0\vert^L-\sum_{i=0}^{\nu}\gamma_0^{n+m+2}s_i^{m-1}\mu_i>\frac{\sigma\delta^L}{2},
\end{align*}
where $\delta:=\min\{\vert y-\xi_0\vert, \forall y\in\partial O\}$.

So, it follows from homotopy invariance and $\rm(A0)$ that
\begin{align*}
\deg\left(H_1(\cdot),O^o,0\right)=\deg\left(H_0(\cdot),O^o,0\right)\neq0.
\end{align*}
We note by $\rm(A1)$, (\ref{eqb44}) and (\ref{equ21}) that for any $y\in O\backslash B_{s\mu^{{1}/{L}}}(\xi)$,
\begin{align*}
\left\vert H_1(y)\right\vert&=\vert\omega_+(y)-\omega(\xi)\vert\geq \vert y-\xi\vert^L-c_1\gamma_0^{n+m+2}s^{m-1}\mu\\
&\geq  s^L\mu-c_1\gamma_0^{n+m+2}s^{m-1}\mu\geq\frac{s^L\mu}{2}.
\end{align*}
Hence, by excision, we have that
\begin{align*}
\deg\left(H_1(\cdot),B_{s\mu^{{1}/{L}}}(\xi),0\right)=\deg\left(H_1(\cdot),O^o,0\right)\neq0,
\end{align*}
i.e., there exists at least a $\xi_+\in B_{s\mu^{{1}/{L}}}(\xi)$, such that
$H_1(\xi_+)=0,$
i.e., $$\omega_+(\xi_+)=\omega(\xi),$$ which implies (\ref{equ22}).

The proof is complete.
\end{proof}

In the following, we prove
\begin{align}\label{eqc2}
\phi:D_{\frac{1}{8}\alpha}\rightarrow D_{\frac{1}{4}\alpha},
\end{align}
 which is different from (\ref{eqb18}) in Lemma \ref{leb6}.
Recall that $m>L+1$ and $\alpha=\mu^{\frac{1}{m+1}}$, we have
\begin{align}\label{eqc1}
cs\mu^{\frac{1}{L}}<\frac{1}{8}\alpha s.
\end{align}
For $\forall (y,x)\in D_{\frac{1}{8}\alpha}$, we note by $\xi_+\in B_{s\mu^{{1}/{L}}}(\xi)$ in Lemma \ref{lec5} and (\ref{eqc1}) that
$$\left\vert y+\xi_+-\xi\right\vert <\vert y\vert+\vert \xi_+-\xi\vert <\frac{1}{8}\alpha s+cs\mu^{\frac{1}{L}}<\frac{1}{4}\alpha s,$$ which implies (\ref{eqc2}).

Next, we prove Theorem \ref{th3} by a direct method.
\subsection{Proof of Theorem \ref{th3}}
(1) The unperturbed motion of (\ref{eq1}) is described by the equation
\begin{equation*}
\left\{
\begin{array}{ll}
\dot{y}=0,\\
\dot{x}= h'(y).
\end{array}
\right.
\end{equation*}
The flow is $x=h'(y)t+x_0,y\in G$, where $x_0$ is an initial condition. Notice that
\begin{align*}
h''(0)=0,
\end{align*}
i.e., $h(y)$ is degenerate at $\xi_0=0$. Obviously, by simple calculation, we get
\begin{align*}
\deg\left(h'(y)-h'(0),B_\delta(0),0\right)=0,
\end{align*}
i.e., $\rm(A0)$ fails,
then Theorem \ref{pro2} is not applicable.

Note that the perturbed motion equation is
\begin{equation*}
\left\{
\begin{array}{ll}
\dot{y}=0,\\
\dot{x}=h'(y)+\varepsilon P'(y).
\end{array}
\right.
\end{equation*}
The flow is $x=\left(h'(y)+\varepsilon P'(y)\right)t+x_1,~~~y\in G$, where $x_1$ is an initial condition. To ensure the frequency is equal to $h'(0)$, we need to find a solution of the following equation in $G$:
\begin{align*}
h'(y)+\varepsilon P'(y)=h'(0),
\end{align*}
i.e.,
\begin{align}\label{eqa4}
g'(y)=-\varepsilon P'(y).
\end{align}
Notice that the Taylor expansion of $g'(y)$ at $\xi_0=0$ is
$$g'(y)=g'(0)+g''(0)y+\cdots+g^{2\ell+1}(0)y^{2\ell}+o(y^{2\ell}),$$ then the equation (\ref{eqa4}) is equivalent to
$$g^{2\ell+1}(0)y^{2\ell}+o(y^{2\ell})=-\varepsilon P'(y),$$
which is solvable provided that $\varepsilon P'(y)\, sign \left(g^{2\ell+1}(0)\right)<0$. So the perturbed system admits at least two invariant tori with frequency $\omega=h'(0)$ for the small enough perturbation satisfying $\varepsilon P'(y)\, sign \left(g^{2\ell+1}(0)\right)<0$.
Conversely, if $\varepsilon P'(y)\, sign \left(g^{2\ell+1}(0)\right)>0$, the unperturbed invariant torus with frequency $\omega=h'(0)$ will be destroyed.

(2)
Note that $h(y)$ is degenerate in $\xi_0=0$. Obviously, by simple calculation, we get
\begin{align*}
\deg\left(h'(y)-h'(0),B_\delta(0),0\right)\neq0.
\end{align*}
Then, by Theorem \ref{pro2}, the above persistence result hold. In addition, we can also directly prove this result. Similarly, we need to solve the following equation in $G$:
\begin{align*}
h'(y)+\varepsilon P'(y)=h'(0),
\end{align*}
i.e.,
\begin{align}\label{eqa5}
g'(y)=-\varepsilon P'(y).
\end{align}
Notice that the Taylor expansion of $g'(y)$ at $\xi_0=0$ is
$$g'(y)=g'(0)+g''(0)y+\cdots+g^{2\ell+2}(0)y^{2\ell+1}+o\left(y^{2\ell+1}\right),$$ then the equation (\ref{eqa5}) is equivalent to
$$g^{2\ell+2}\left(0\right)y^{2\ell+1}+o\left(y^{2\ell+1}\right)=-\varepsilon P'(y),$$
whose solution always exists in $G$ for any small enough perturbation. Hence, the perturbed system admits an invariant torus with frequency $\omega=h'(0)$ for any small enough perturbation.


%
%
%

\section{Appendix A. Proof of Proposition \ref{pro1}}\label{appa}

\begin{proof}
Obviously, for $\forall \xi\in (-1,1)\times(-1,1)$,
$$(\omega-\bar{\omega})(-\xi)=-(\omega-\bar{\omega})(\xi),$$
and for $\forall \xi\in\partial(-1,1)\times(-1,1),$
$$(\omega-\bar{\omega})(\xi)\neq0.$$
Using Borsuk's theorem in \cite{guo}, we have
$$\deg\left(\omega(\cdot)-\bar{\omega},(-1,1)\times(-1,1),0\right)\neq0,$$
i.e.,
$$\deg\left(\omega(\cdot),(-1,1)\times(-1,1),\bar{\omega}\right)\neq0,$$ i.e., $\rm(A0)$ holds. For $\xi,\xi_*\in\left[-\frac{1}{2},\frac{1}{2}\right]$, and $\xi\neq\xi_*$, we have $$\omega(\xi)-\omega(\xi_*)=0,$$ but $$\left\vert\xi-\xi_*\right\vert^L>0,~~\forall L>0,$$ which shows that $\rm(A1)$ fails.
Note that the flow of unperturbed motion equation is $$x=\omega(\xi)t+x_0,~~~\xi\in (-1,1)\times(-1,1),$$ where $x_0$ is an initial condition, and the flow of perturbed motion equation is
$$x=\left(\omega(\xi)+\left(0,P_0\left(\varepsilon\right)\right)^\top\right)t+x_0,~~~\xi\in (-1,1)\times(-1,1).$$ In order to keep the frequency $\omega(0)=\bar\omega$ unchanged, we have to solve the following equation $$\omega(\xi)+\left(0,P_0\left(\varepsilon\right)\right)^\top=\bar\omega,$$ i.e.,
$$\omega\left(\xi\right)-\bar\omega\left(\xi\right)=\left(\xi_1, \omega_2-\bar\omega_2\right)^\top=-\left(0,P_0\left(\varepsilon\right)\right)^\top,$$ which implies that the second component $\xi_2$ of solution $\xi$ is discontinuous and alternately appears on $\left(-1,-\frac{1}{2}\right)$ and $\left(\frac{1}{2},1\right)$ as $\varepsilon\rightarrow0_+$. So, this example shows that condition $\rm(A1)$ is necessary no matter how smooth the frequency mapping $\omega(\xi)$ is.
\end{proof}

\section{\textbf{Appendix B. Proof of Theorem \ref{pro2}}}\label{appb}
\begin{proof}
Notice that
\begin{align*}
\nabla h(y)-\nabla h(0)=y\vert y\vert^{2l}.
\end{align*}
For $0<\delta<1$, $B_\delta(0)$ denotes the open ball centered at the origin with radius $\delta$. We have that $\nabla h(y)-\nabla h(0)$ is odd and unequal to zero on $\partial B_\delta(0)$, i.e.,
$$\nabla h(-y)-\nabla h(0)=-y\vert y\vert^{2l}=-(\nabla h(y)-\nabla h(0)),~~\nabla h(y)-\nabla h(0)\neq0,~~\forall y\in\partial B_\delta(0).$$
It follows from Borsuk's theorem in \cite{guo} that,
\begin{align*}
\deg(\nabla h(y)-\nabla h(0),B_\delta(0),0)\neq0.
\end{align*}
Obviously, there exist $\sigma=\frac{\min_{y\in B_\delta(0)}\{{(y+\varrho)\vert y+\varrho\vert^{2\ell}-y\vert y\vert^{2\ell}}\}}{2\varrho^{2\ell+1}}$ and $L=2l+1$ such that
\begin{align*}
\left\vert\nabla h(y)-\nabla h(y_*)\right\vert\geq\sigma\left\vert y-y_*\right\vert^L,~~y_*\in B_\delta(0), y\in B_\delta(0)\setminus B_\varrho(y_*),
\end{align*}
where~$\varrho>0$, $B_\varrho(y_*)\subset B_\delta(0)$.
So, by Theorem \ref{th2}, the perturbed system admits an invariant torus with frequency $\omega$ for any small enough perturbation.
\end{proof}

\section{\textbf{Appendix C. Proof of Proposition \ref{cor1}}}\label{appc}

\begin{proof}
Let $\varepsilon P=\varepsilon y, \varepsilon>0$. Notice that for $y\in G\subset\mathbb{R}^1$,
$$h'(y)=\omega+y^{2\ell},~~h'(0)=\omega,~~h''(y)\big\vert_{y=0}=0,$$
which implies that the Hamiltonian $H$ is degenerate at $y=0$. By the definition of degree, we have for $0<\delta<1$
\begin{align*}
\deg\left(\nabla h(y)-\nabla h(0),B_\delta(0),0\right)=0,
\end{align*}
i.e., $\rm(A0)$ fails. Then, Theorem \ref{th2} cannot be used to prove the persistence result of keeping frequency unchanged.

Note that the flow of unperturbed motion equation at $y=0$ is $$x=\omega t+x_0,$$ where $x_0$ is an initial condition, and the flow of perturbed motion equation is $$x=\left(\omega+y^{2\ell}+\varepsilon\right)t+x_0,~~~y\in G.$$ In order to preserve frequency $\omega$, we need to solve $y^{2\ell}+\varepsilon=0$ in $G$, which has no real solution in $G$. Hence, the persistence result of keeping frequency unchanged fails.

\end{proof}

\section*{Acknowledgments}

The second author (Li Yong) is supported by National Basic Research Program of China (Grant number
[2013CB8-34100]), National Natural Science Foundation of China (Grant numbers [11571065], [11171132], and
[12071175]),
and Natural Science Foundation of Jilin Province (Grant number [20200201253JC]).


\end{document}